\documentclass{amsart}
\usepackage[utf8]{inputenc}
\usepackage{amsmath, amssymb, bbm}
\usepackage{amsthm}
\usepackage{mathtools}
\usepackage{blkarray}
\usepackage[matrix,arrow,curve]{xy}
\usepackage{color}
\usepackage{tikz}
\usepackage{graphicx}
\definecolor{darkblue}{rgb}{0,0,0.6}
\usepackage[ocgcolorlinks,colorlinks=true, citecolor=darkblue, filecolor=darkblue, linkcolor=darkblue, urlcolor=darkblue]{hyperref}
\usepackage[capitalize,noabbrev]{cleveref}
\usepackage{enumerate}

\newcommand{\ignore}[1]{}
\renewcommand{\epsilon}{\varepsilon}
\renewcommand{\phi}{\varphi}
\newcommand{\Z}{\mathbbm{Z}}

\newcommand{\Q}{\mathbbm{Q}}
\newcommand{\R}{\mathbbm{R}}
\newcommand{\RP}{\mathbb{RP}^2}
\newcommand{\CP}{\mathbb{CP}^2}
\newcommand{\bCP}{\overline{\mathbb{CP}}^2}

\newcommand{\bbN}{\mathbbm{N}}
\newcommand{\N}{\mathbbm{N}}
\newcommand{\La}{\Lambda}

\newcommand{\ra}{\longrightarrow}
\newcommand{\hra}{\hookrightarrow}
\newcommand{\toiso}{\xrightarrow{\cong}}
\newcommand{\raiso}{\xrightarrow{\,\cong\, }}
\newcommand{\wt}{\widetilde}
\newcommand{\ol}{\overline}

\newcommand{\Ext}{\mathrm{Ext}}

\newcommand{\sm}{\setminus}

\newcommand{\scong}{{\cong_s}}
\newcommand{\wh}{\widehat}
\newcommand{\ba}{\begin{array}}
	\newcommand{\ea}{\end{array}}

\DeclareMathOperator{\Aut}{Aut}

\DeclareMathOperator{\Hom}{Hom}
\DeclareMathOperator{\id}{Id}
\DeclareMathOperator{\im}{im}
\DeclareMathOperator{\coker}{coker}
\DeclareMathOperator{\Id}{Id}

\DeclareMathOperator{\cl}{cl}

\DeclareMathOperator{\pr}{pr}

\DeclareMathOperator{\rank}{rank}
\DeclareMathOperator{\Diff}{Diff}

\DeclareMathOperator{\sAut}{sAut}
\DeclareMathOperator{\hCh}{hCh}
\DeclareMathOperator{\sCh}{sCh}
\DeclareMathOperator{\hCW}{hCW}
\DeclareMathOperator{\ks}{ks}

\newcommand{\bsm}{\left(\begin{smallmatrix}}
\newcommand{\esm}{\end{smallmatrix}\right)}

\numberwithin{equation}{section}
\newtheorem{thm}[equation]{Theorem}
\newtheorem{theorem}[equation]{Theorem}

\newtheorem{prop}[equation]{Proposition}

\newtheorem{cor}[equation]{Corollary}
\newtheorem{lemma}[equation]{Lemma}
\newtheorem{proposition}[equation]{Proposition}

\newtheorem{question}[equation]{Question}
\newtheorem{theoremalph}{Theorem}

\theoremstyle{definition}

\newtheorem{example}[equation]{Example}

\newtheorem{definition}[equation]{Definition}
\newtheorem{rem}[equation]{Remark}
\newtheorem{remark}[equation]{Remark}

\newtheorem*{claim}{Claim}

\crefname{lemma}{Lemma}{Lemmas}
\crefname{section}{Section}{Sections}
\crefname{definition}{Definition}{Definitions}
\crefname{defi}{Definition}{Definitions}
\crefname{prop}{Proposition}{Propositions}
\crefname{theorem}{Theorem}{Theorems}
\crefname{thm}{Theorem}{Theorems}
\crefname{cor}{Corollary}{Corollaries}
\crefname{rem}{Remark}{Remarks}

\title[$\CP$-stable classification of $4$-manifolds]{Four-manifolds up to connected sum with complex projective planes}

\author{Daniel Kasprowski}
\address{Rheinische Friedrich-Wilhelms-Universit\"{a}t Bonn, Mathematisches Institut,\newline \indent Endenicher Allee 60, 53115 Bonn, Germany}
\email{kasprowski@uni-bonn.de}

\author{Mark Powell}
\address{Department of Mathematical Sciences, Durham University, Upper Mountjoy,\newline \indent Stockton Road, Durham, DH1 3LE}
\email{mark.a.powell@durham.ac.uk}

\author{Peter Teichner}
\address{Max-Planck-Institut f\"{u}r Mathematik, Vivatsgasse 7, 53111 Bonn, Germany}
\email{teichner@mpim-bonn.mpg.de}

\makeatletter
\@namedef{subjclassname@2020}{\textup{2020} Mathematics Subject Classification}
\makeatother
\subjclass[2020]{57N13, 57N65}
\keywords{4-manifolds, $\CP$-stable classification, Postnikov 2-type, $k$-invariant}

\begin{document}

\begin{abstract}
Based on results of Kreck, we show that closed, connected $4$-manifolds up to connected sum with copies of the complex projective plane are classified in terms of the fundamental group, the orientation character and an extension class involving the second homotopy group.  For fundamental groups that are torsion free or have one end, we reduce this further to a classification in terms of the homotopy 2-type.
\end{abstract}

\maketitle
\section{Introduction}

We give an explicit, algebraic classification of closed, connected 4-manifolds up to connected sum with copies of the complex projective plane $\CP$.

After the great success of Thurston's geometrisation of 3-manifolds, the classification of closed 4-manifolds remains one of the most exciting open problems in topology. The exactness of the surgery sequence and the $s$-cobordism theorem are known for topological $4$-manifolds with good fundamental groups, a class of groups that includes all solvable groups~\cite{FT, KQ}. However, a homeomorphism classification is only known for closed $4$-manifolds with trivial~\cite{Freedman82}, cyclic~\cite{Freedman-Quinn, surgeryandduality, hambleton-kreck-93-III} or Baumslag-Solitar~\cite{HKT}  fundamental group.

For smooth 4-manifolds, gauge theory provides obstructions even to once hoped-for foundational results like simply connected surgery and $h$-cobordism, which hold in all higher dimensions. There is no proposed diffeomorphism classification in sight, indeed understanding homotopy 4-spheres is beyond us at present.
Most of the invariants derived from gauge theory depend on an orientation and do not change under connected sum with $\bCP$, but the differences dissolve under connected sum with $\CP$.  This suggests considering $4$-manifolds up to \emph{$\CP$-stable diffeomorphism}.

In Section~\ref{sec:definitions} we spell out all the details, and discuss some history, but first we would like to state our main result.
The \emph{$2$-type} of a connected manifold $M$ consists of its $1$-type $(\pi_1(M) ,w_1(M))$ together with the second homotopy group $\pi_2(M)$ and the $k$-invariant $k(M) \in H^3(\pi_1(M) ;\pi_2(M))$ that classifies the second Postnikov stage of $M$ via a fibration $K(\pi_2(M),2) \to P_2(M) \to K(\pi_1(M),1)$.

\begin{theoremalph}\label{thm:k}
Two closed, connected, smooth $4$-manifolds with $1$-type $(\pi,w)$ are $\CP$-stably diffeomorphic if and only if their $2$-types $(\pi,w,\pi_2,k)$ are stably isomorphic, provided $\pi$ is
\begin{enumerate}[(i)]
	\item\label{thm:k:item1} torsion-free; or
	\item \label{thm:k:item2} infinite with one end; or
	\item \label{thm:k:item3} finite with $H_4(\pi;\Z^w)$ annihilated by $4$ or $6$.
\end{enumerate}
\end{theoremalph}

Definition~\ref{def:k-stable} discusses the notion of a stable isomorphism of $2$-types.
Condition \eqref{thm:k:item3} is satisfied for all $(\pi,w)$ with $\pi$ finite and $w \colon \pi\to\{\pm 1\}$ nontrivial on the centre of $\pi$, as explained in Remark~\ref{rem:centre}. A curious consequence of Theorem~\ref{thm:k} is worth pointing out.

\begin{cor}
Let $M_1$ and $M_2$ be two closed, connected, smooth $4$-manifolds with fundamental group $\pi$ as in Theorem~\ref{thm:k}. If $M_1$ and $M_2$ have stably isomorphic 2-types, then the equivariant intersection forms on $\pi_2(M_1)$ and $\pi_2(M_2)$ become isometric after stabilisation by standard forms $(\pm 1)$.
\end{cor}

In the simply connected case, this follows from the classification of odd indefinite forms by their rank and signature, since for two given simply connected 4-manifolds, the rank and signatures of their intersection forms can be made to coincide by $\CP$-stabilisation. For general fundamental groups, the underlying module does not algebraically determine the intersection form up to stabilisation, but \cref{thm:k} says that equivariant intersection forms of $4$-manifolds with the appropriate fundamental group are controlled in this way.

\subsection{Definitions of various notions of stability}\label{sec:definitions}

The connected sum of a smooth 4-manifold $M$ with $\CP$ depends on a choice of embedding $D^4 \hra M$, where isotopic embeddings yield diffeomorphic connected sums. A complex chart gives a preferred choice of isotopy class of embeddings $D^4 \subset \mathbb{C}^2 \hra \CP$. We will assume that the manifolds we consider are connected unless stated otherwise.  If $M$ admits an orientation $\mathfrak{o}$, there are two distinct isotopy classes of embeddings $D^4 \hra M$, exactly one of which is orientation preserving. So we obtain two manifolds $(M,\mathfrak{o})\#\CP$ and $(M,-\mathfrak{o})\#\CP \cong (M,\mathfrak{o})\#\bCP$  that in general are not diffeomorphic.  On the other hand if $M$ is not orientable, we can isotope an embedding $D^4 \hra M$ around an orientation reversing loop to see that there is an essentially unique connected sum~$M\#\CP$.

While our theorems are stated for unoriented 4-manifolds, we will often make some choice of a (twisted) fundamental class of~$M$. But for our results the specific choice will not matter since we will usually factor out by the effect of this choice.

A \emph{$1$-type} $(\pi,w)$ consists of a group $\pi$ and a homomorphism $w \colon \pi  \to \{\pm 1\}$.  A connected manifold $M$ has 1-type $(\pi,w)$ if there is a map $c \colon M \to B\pi$ such that  $c_* \colon \pi_1(M)  \to \pi_1(B\pi) = \pi$ is an isomorphism and $w \circ c_* \colon \pi_1(M)  \to\{\pm 1\}$ determines the first Stiefel-Whitney class $w_1(M)  \in H^1(M;\Z/2)\cong \Hom(\pi_1(M) ,\Z/2)$.
All our results are based on the following theorem.

\begin{theorem}[Kreck {\cite{surgeryandduality}}]\label{theorem:Kreck}
Two closed, connected, smooth $4$-manifolds $M_1$ and $M_2$ are $\CP$-stably diffeomorphic if and only if they have the same $1$-type $(\pi,w)$ and the images of some choices of $($twisted$)$ fundamental classes $(c_1)_*[M_1]$ and $(c_2)_*[M_2]$ coincide in $H_4(\pi;\Z^w)/\pm\Aut(\pi)$.
\end{theorem}

The map induced on homology by the classifying map $c \colon M \to B\pi$ sends the (twisted) fundamental class $[M]\in H_4(M;\Z^{w_1(M)})$ to $H_4(\pi;\Z^w)$.  Here the coefficients are twisted using the orientation characters $w_1(M)$ and $w$ to give $\Z^{w_1(M)}$ and $\Z^w$ respectively.   The quotient by $\Aut(\pi)$ takes care of the different choices of identifications $c_*$ of the fundamental groups with $\pi$, and the sign $\pm$ removes dependency on the choice of fundamental class.
In particular, within a 1-type $(\pi,w)$ such that $H_4(\pi;\Z^w)=0$, there is a single $\CP$-stable diffeomorphism class.
Kreck's result also implies that if $M_1$ and $M_2$ are homotopy equivalent then they are $\CP$-stable diffeomorphic.
We will give a self-contained proof of \cref{theorem:Kreck} in \cref{sec:stable classification}, that proceeds by simplifying a handle decomposition.

\begin{rem}\label{rem:topological}
Kreck's theorem, as well as our results, also hold for topological $4$-manifolds, with the additional condition that the Kirby-Siebenmann invariants $\ks(M_i) \in\Z/2$ coincide for $i=1,2$ cf.~\cite[Chapter~5]{teichnerthesis}.  A topological $4$-manifold $M$ has a smooth structure $\CP$-stably if and only if $\ks(M)=0$.  Using the existence of simply-connected closed 4-manifolds with non-trivial Kirby-Siebenmann invariant, one can quickly deduce the topological result from the smooth one.  See \cref{remark:topological-case} for more details.   We therefore focus on the smooth category from now on.
\end{rem}

A connected sum with $\CP$ changes the second homotopy group by adding a free summand $\pi_2(M \# \CP) \cong \pi_2(M)  \oplus \La$, where here and throughout the paper we write $[\La := \Z[\pi_1(M) ] \cong \Z[\pi]$ for the group ring.   The $k$-invariant $k(M) \in H^3(\pi;\pi_2(M))$ maps via $(k(M),0)$ to $k(M\# \CP)$ under the composition
\[
H^3(\pi;\pi_2(M) ) \ra H^3(\pi;\pi_2(M) ) \oplus H^3(\pi;\La) 
\raiso H^3(\pi;\pi_2(M \# \CP)).
\]

\begin{definition}\label{def:k-stable}
A \emph{stable isomorphism} of $2$-types is a pair $(\varphi_1, \varphi_2)$ consisting of an isomorphism $\varphi_1 \colon \pi_1(M_1) \to \pi_1(M_2)$ with $\varphi_1^*(w_1(M_2)) = w_1(M_1)$, together with an isomorphism, for some $r,s\in\N_0$,
\[\varphi_2 \colon  \pi_2(M_1) \oplus \La^r \raiso \pi_2(M_2) \oplus \La^s \text{ satisfying } \varphi_2(g\cdot x) = \varphi_1(g)\cdot \varphi_2(x) \]
for all $g\in\pi_1(M_1)$ and for all $x\in\pi_2(M_1) \oplus \La^r$. We also require that $(\varphi_1, \varphi_2)$ preserves $k$-invariants in the sense that
\[
\begin{array}{rcl} (\varphi_1^{-1}, \varphi_2) \colon  H^3(\pi_1(M_1);\pi_2(M_1) \oplus \La^r) &\ra & H^3(\pi_1(M_2);\pi_2(M_2) \oplus \La^s) \\
(k(M_1),0)  &\mapsto & (k(M_2),0). \end{array}
\]
\end{definition}

\begin{rem}\label{rem:centre}
Observe that, by design, a $\CP$-stable diffeomorphism induces a stable isomorphism of 2-types, so the `only if' direction of \cref{thm:k} holds for all groups.
In addition, for an arbitrary fundamental group $\pi$, the stable $\CP$-stable diffeomorphism class represented by those manifolds with $c_*[M]=0\in H_4(\pi;\Z^w)$ is always detected by the stable isomorphism class of their 2-type $(\pi,w,\pi_2,k)$, cf.\ \cref{cor:detect-trivial}.

Condition \eqref{thm:k:item3} in Theorem~\ref{thm:k} is satisfied for all $(\pi,w)$ with $\pi$ finite and $w$ nontrivial on the centre of $\pi$.  To see this, apply \cite[Proposition~III.8.1]{brown} for some element $g$ in the centre of $\pi$ with $w(g)$ nontrivial, to show that multiplication by $-1$ acts as the identity. Hence every nontrivial element in $H_4(\pi;\Z^w)$ has order two.
\end{rem}

\subsection{Necessity of assumptions}\label{sec:necessity-intro}

Next we present examples of groups demonstrating that hypotheses of \cref{thm:k} are necessary. The details are given in Section~\ref{section:post-2-type-does-not-suffice-lens-spaces}.  We consider a class of infinite groups with two ends, namely $\pi = \Z \times \Z/p$, and closed, orientable $4$-manifolds, so $w=0$.  In this case the 2-type does not determine the $\CP$-stable diffeomorphism classification, as the following example shows.

\begin{example}\label{ex:lens}
Let $L_{p_1,q_1}$ and $L_{p_2,q_2}$ be two $3$-dimensional lens spaces, which are closed, oriented $3$-manifolds with cyclic fundamental group $\Z/p_i$ and universal covering $S^3$.  Assume that $p_i \geq 2$ and $1 \leq q_i < p_i$. The $4$-manifolds $M_i:=S^1 \times  L_{p_i,q_i}$, $i=1,2$, have $\pi_2(M_i) = \{0\}$. Whence their $2$-types are stably isomorphic if and only if $\pi_1(L_{p_1,q_1}) \cong \pi_1(L_{p_2,q_2})$, that is if and only if $p_1=p_2$.
However, we will show that the $4$-manifolds $M_1$ and $M_2$ are $\CP$-stably diffeomorphic if and only if $L_{p_1,q_1}$ and~$L_{p_2,q_2}$ are homotopy equivalent.

It is a classical result that there are homotopically inequivalent lens spaces with the same fundamental group. In fact it was shown by J.H.C. Whitehead that $L_{p,q_1}$ and $L_{p,q_2}$ are homotopy equivalent if and only if their $\Q/\Z$-valued linking forms are isometric.  See ~\cref{section:post-2-type-does-not-suffice-lens-spaces} for more details and precise references.
\end{example}

We do not know an example of a finite group $\pi$ for which the conclusion of \cref{thm:k} does not hold. Thus the following question remains open.

\begin{question}
Does the conclusion of \cref{thm:k} hold for all finite groups?
\end{question}

In \cite{KT}, the first and third authors showed that for $4$-manifolds with finite fundamental group, if one adds the data of the stable class of the equivariant intersection form to the $2$-type, then the $\CP$-stable class is determined.

\subsection{Is the \texorpdfstring{$k$}{k}-invariant required?}\label{sec:k-required-intro}

In \cref{sec:aspherical}, we will show that while \cref{thm:k} applies, the $k$-invariant is not needed for the $\CP$-stable classification of 4-manifolds $M$ with  fundamental group $\pi$, where $\pi$ is also the fundamental group of some closed aspherical $4$-manifold.  A good example of such a $4$-manifold is the 4-torus, with fundamental group $\pi=\Z^4$. More generally a surface bundle over a surface, with neither surface equal to $S^2$ nor $\RP$, is an aspherical 4-manifold.  We know from  \cref{theorem:Kreck} that the $\CP$-stable equivalence classes are in bijection with~$\N_0$ because $H_4(\pi;\Z^w)\cong \Z$, but it is not obvious how to compute this invariant from a given $4$-manifold~$M$.

We will show in \cref{thm:aspherical} that in this case the stable isomorphism type of the $\La$-module $\pi_2(M) $ determines the $\CP$-stable diffeomorphism class of~$M$. Assuming moreover that $H^1(\pi;\Z)\neq 0$, we will show in Theorem~\ref{thm:priexample} that the highest torsion in the abelian group of twisted co-invariants $\Z^w \otimes_{\La} \pi_2(M) $ detects this stable isomorphism class in almost all cases.

On the other hand, there are many cases where \cref{thm:k} applies and the $k$-invariant is indeed required, as the next example shows. This example also leverages the homotopy classification of lens spaces.

\begin{example}\label{ex:k}
For the following class of 4-manifolds, the $k$-invariant is required in the $\CP$-stable classification. Let $\Sigma$ be an aspherical 3-manifold.  Consider a lens space $L_{p,q}$ with fundamental group $\Z/p$ and form the 4-manifold $M(L_{p,q},\Sigma) := S^1\times (L_{p,q}\# \Sigma)$ with fundamental group $\pi=\Z\times (\Z/p * \pi_1(\Sigma))$.

We will show in \cref{sec:k-required} that these groups have one end and hence our Theorem~\ref{thm:k} applies, so the 2-type determines the $\CP$-stable classification.  Similarly to Example~\ref{ex:lens}, we will show that two $4$-manifolds of the form $M(L_{p,q},\Sigma)$ are $\CP$-stably diffeomorphic if and only if the involved lens spaces are homotopy equivalent. However, the $\La$-modules $\pi_2(M(L_{p,q},\Sigma))$ will be shown to depend only on $p$. Since there are homotopically inequivalent lens spaces with the same fundamental group, we deduce that the stable isomorphism class of $\pi_2(M(L_{p,q},\Sigma))$ is a weaker invariant than the full 2-type $(\pi, 0,\pi_2, k)$.
\end{example}

\subsection{Extension classes and the proof of Theorem~\ref{thm:k}}

In order to prove \cref{thm:k}, we will translate the image of the fundamental class $c_*[M] \in H_4(\pi;\Z^w)$ completely into algebra.

Let $(C_*,d_*)$ and $(C'_*,d'_*)$ be free resolutions of $\Z$ as a $\La$-module with $C_i$ and $C'_i$ finitely generated for $i=0,1,2$. See \cref{lem:naturality-ext} for an explanation of the naturality in the following proposition. Note that $d^{',2}_w$ is the dual of $d'_2$ twisted by the orientation character $w$, see \cref{sec:conventions} for our conventions.

\begin{proposition} \label{prop:mainpri1}
	There is a natural isomorphism
	\[\psi(C_*,C'_*) \colon H_4(\pi;\Z^w)\raiso \Ext_{\La}^1\left(\coker d_w^{\prime,2},\ker d_2\right). \]
\end{proposition}

As always, $M$ is a closed, connected, smooth $4$-manifold together with a $2$-equivalence $c \colon M\to B\pi$ such that $c^*(w) = w_1(M)$.  We choose a (twisted) fundamental class and a handle decomposition of $M$ and consider the chain complex $(C_*^M,d^M_*)$ of left $\La$-modules, freely generated by the handles in the universal cover~$\widetilde{M}$. Note that by the Hurewicz theorem, $\pi_2(M)  \cong H_2(\widetilde M) \cong \ker d_2/\im d_3$. Let $M^\natural$ be $M$ with the dual handle decomposition and denote its $\La$-chain complex by $(C_*^\natural,d_*^\natural)$. Then by Poincar\'e duality $C^*_{\natural,w}\cong C_*^M$. In particular, $\coker d^2_{\natural,w}\cong \coker d_3$ and we will use this isomorphism implicitly in \cref{prop:mainpri2} below.  Note that picking the other orientation of $M$ changes this isomorphism by a sign. Thus in the next proposition the image of $c_*[M]$ in $\Ext_{\La}^1(\coker d_3, \ker d_2)$ is independent of the choice of $[M]$.

\begin{proposition}\label{prop:mainpri2}
	Let $D_*$ be a free resolution of $\Z$ starting with $D_i=C_i^M$ for $i=0,1,2$ and let $D'_*$ be a free resolution of $\Z$ starting with $D'_i=C_i^\natural$ for $i=0,1,2$.
	The isomorphism $\Psi(D_*,D_*')$ from \cref{prop:mainpri1} sends $c_*[M]$ to the extension
\[0\to \ker d_2 \xrightarrow{(i,-p)^T} C_2\oplus \ker d_2/\im d_3\xrightarrow{(p',i')} \coker d_3\to 0\]
where $i,i'$ are the inclusions and $p,p'$ are the projections.  In particular, $c_*[M]=0$ if and only if this extension is trivial, and hence $\pi_2(M) $ is stably isomorphic to the direct sum $\ker d_2\oplus \coker d_3$.
\end{proposition}

\cref{prop:mainpri2} is a generalisation of \cite[Proposition~2.4]{Ham-kreck-finite}, where the theorem was proven for oriented manifolds with finite fundamental groups.
The fact that~$\pi_2(M) $ fits into such an extension for general groups was shown by Hambleton in~\cite{hambleton-gokova}.
The above proposition implies the following, because $c_*[M] \in H_4(\pi;\Z^w)$ determines the $\CP$-stable diffeomorphism type by \cref{theorem:Kreck}.

\begin{thm}\label{thm:prim-obstruction}
Two closed, connected, smooth $4$-manifolds $M_1$ and $M_2$ with $1$-type $(\pi,w)$ are $\CP$-stably diffeomorphic if and only if the extension classes determining $\pi_2(M_1)$ and $\pi_2(M_2)$ coincide in
\[
\Ext^1_{\La}(\coker d^2_w,\ker d_2) \cong H_4(\pi;\Z^w)
\]
modulo the action of $\pm\Aut(\pi)$.
\end{thm}

This result translates the $H_4(\pi;\Z^w)$ invariant into algebra for all groups.
There is a version of \cref{thm:prim-obstruction} where we do not need to divide out by $\Aut(\pi)$: there is a stable diffeomorphism \emph{over $\pi$} if the extension classes agree (up to sign) for a specific choice of identifications $\pi_1(M_i) \xrightarrow{\cong} \pi$.

In \cref{sec:k-ext1} and \cref{sec:k-ext}, we will derive \cref{thm:k} from \cref{thm:prim-obstruction}. Since two isomorphic extensions yield stably isomorphic second homotopy groups, we seek a kind of converse by adding the datum of the $k$-invariant.
In the purely algebraic \cref{sec:k-ext1}, given a 1-cocycle $f \colon C^1_w\to \ker d_2$ representing an extension class in $\Ext^1_{\La}(\coker d^2_w,\ker d_2)$, we construct a $\La$-module chain complex $C_f$ by attaching (trivial) 2-chain and (non-trivial) 3-chains to~$C_{2} \to C_1 \to C_0$. We show that $\Theta \colon f\mapsto C_f$ gives a well-defined map from our extension group $\Ext^1_{\La}(\coker d^2_w,\ker d_2)$ to stable isomorphism classes in a category $\sCh_2(\pi)$ of chain complexes over $\La=\Z\pi$ that are the algebraic analogue of the 2-type. Moreover, if $f$ is the extension class coming from a 4-manifold~$M$, then $\Theta(f) = C_f$ agrees in $\sCh_2(\pi)$ with the stable class determined by the chain complex of $P_2(M)$ of~$M$. Recall that $P_2(M)$ is determined by 2-type of $M$, which includes the $k$-invariant.
In Theorem~\ref{thm:injective} we then show that the stable class of $C_f$ detects our extension class modulo the action of the automorphisms of $\coker d^2_w$.  That is, $\Theta$ induces an injective map on the quotient of the extensions by the automorphisms of $\coker d^2_w$.  Finally, in \cref{lem:autcoker} we analyse the action of such automorphisms on the extension group $\Ext^1_{\La}(\coker d^2_w,\ker d_2)$ and show that under the assumptions of Theorem~\ref{thm:k}, the automorphisms can change the extension class at most by a sign.  Since our manifolds are unoriented, the sign ambiguity is already present, so \cref{thm:k} follows.

\subsection{Additional remarks}

\subsubsection*{Topological manifolds}\label{remark:topological-case}
As mentioned in \cref{rem:topological}, \cref{thm:k} and \cref{thm:prim-obstruction} hold for closed, connected topological 4-manifolds $M_i$, with the additional assumption that the Kirby-Siebenmann invariants $\ks(M_i) \in \Z/2$ are equal. Here is the proof, which is well-known to the experts.

Let $M_1$ and $M_2$ be two closed, connected, topological 4-manifolds, with $\ks(M_1)=\ks(M_2)$, and suppose that the conditions of one of the two theorems mentioned above are satisfied. Note that all the conditions involve algebraic topological invariants, so are category independent.  As explained in \cite[Section~8]{FriedlNagelOrsonPowell}, it follows from Kirby-Siebenmann~\cite[pp.~321--2]{KirbySiebenmann} and Freedman-Quinn~\cite[Sections~8.6]{Freedman-Quinn} that a $4$-manifold $M$ admits a smooth structure $S^2 \times S^2$-stably if and only if $\ks(M)=0$. Since $S^2 \times S^2$ and $S^4$ are $\CP$-stably diffeomorphic, we deduce that there exists a smooth structure on $M$ $\CP$-stably if and only if $\ks(M)=0$.
To apply this, if necessary connect sum both $M_1$ and $M_2$ with the Chern manifold $*\CP$ with nontrivial Kirby-Siebenmann invariant~\cite[Section~10.4]{Freedman-Quinn}, noting that $\ks$ is additive~\cite[Section~10.2B]{Freedman-Quinn},~\cite[Section~8]{FriedlNagelOrsonPowell}, to obtain a $\CP$-stably smoothable pair. Then the smooth result applies, so that the manifolds are $\CP$-stably diffeomorphic.  Now add another copy of $*\CP$ to both: by the classification of simply-connected 4-manifolds~\cite[Theorem~10.1]{Freedman-Quinn} we have
$*\CP \# *\CP \approx \CP \# \CP,$
so this returns us to the original $\CP$-stable homeomorphism classes, which we now know coincide.

\subsubsection*{Multiplicative Invariants}
Consider an invariant $I$ of closed, oriented $4$-manifolds valued in some commutative monoid, that is multiplicative under connected sum and invertible on $\CP$ and $\bCP$.   For example, the generalised dichromatic invariant of \cite{Barenz-Barrett-17}, $I(M) \in \mathbb{C}$, is such an invariant.

It follows from \cref{theorem:Kreck} that every such invariant is determined by the fundamental group $\pi_1(M) $, the image of the fundamental class in $H_4(\pi;\Z)$, its signature $\sigma(M)$ and its Euler characteristic $\chi(M)$.
More precisely, given a second manifold $N$ with the same fundamental group and $c_*[N]=c_*[M]\in H_4(\pi;\Z)$, one has
\[I(N)=I(M) \cdot I\big(\CP\big)^{\frac{\Delta\chi + \Delta\sigma}{2}} \cdot I\big(\bCP\big)^{\frac{\Delta\chi - \Delta\sigma}{2}},\]
where $\Delta\sigma:= \sigma(N) - \sigma(M)$ and $\Delta\chi := \chi(N) - \chi(M)$.
For example, for every manifold $N$ with fundamental group $\Z$ we have
\[I(N)=I(S^1\times S^3)\cdot I(\CP)^{\frac{\chi(N)+\sigma(N)}{2}}\cdot I(\bCP)^{\frac{\chi(N)-\sigma(N)}{2}}.\]
For the generalised dichromatic invariant, the values on $\CP$, $\bCP$ and $S^1 \times S^3$ are calculated in \cite[Sections~3.4~and~6.2.1]{Barenz-Barrett-17}.
Moreover, in the cases that \cref{thm:k} holds, any invariant as above is equivalently determined by the 2-type, the signature and the Euler characteristic. See~\cite{Reutter} for a comprehensive extension of this discussion.

\subsubsection*{Appearance of $\CP$-stable diffeomorphisms in the literature}
Manifolds up to $\CP$-stable diffeomorphism are also considered by Khan and Smith in \cite{khan-smith}. There the existence of incompressible embeddings of $3$-manifolds corresponding to amalgamated products of the fundamental group is studied. Note that Khan and Smith use the term bistable instead of $\CP$-stable.

\subsubsection*{Relation to the symmetric signature}
We also note that the proofs of \cref{sec:pri,sec:k-ext1} comprise homological algebra, combined with Poincar\'{e} duality to stably identify $\coker(d_3)$ with $\coker d^2_{\natural,w}$, during the passage from \cref{prop:mainpri1} to \cref{thm:prim-obstruction}.  The proofs could therefore be carried out if one only retained the symmetric signature in $L^4(\La,w)$ of the 4-manifold \cite{Ranicki-ATS-I, Ranicki-ATS-II}, that is the chain cobordism class of a $\La$-module handle chain complex of $M$ together with chain-level Poincar\'{e} duality structure.

\subsubsection*{Organisation of the paper}
\cref{sec:stable classification} gives a self-contained proof of Kreck's \cref{theorem:Kreck}.
After establishing conventions for homology and cohomology with twisted coefficients in \cref{sec:conventions}, we present an extended Hopf sequence in
\cref{section:Hopf-sequence}, and use this to give a short proof of \cref{thm:k} in a special case.

\cref{sec:pri} proves \cref{prop:mainpri1,prop:mainpri2}, which together express the fourth homology invariant of \cref{theorem:Kreck} in terms of an extension class involving $\pi_2(M) $.
\cref{sec:k-ext1} refines this in terms of the 2-type for certain groups. This is applied in the  proof of \cref{thm:k} in \cref{sec:k-ext}.

We then give a further refinement in \cref{sec:aspherical}: in the special case that $\pi$ is the fundamental group of some aspherical 4-manifold, the $\CP$-stable classification is determined by the stable isomorphism class of $\pi_2(M) $, and if in addition $H^1(\pi;\Z) \neq 0$ then the classification can essentially be read off from $\Z^w \otimes_{\La} \pi_2(M) $.

Finally, \cref{section:examples-necessity} discusses the examples mentioned in \cref{sec:necessity-intro,sec:k-required-intro}, showing first that the hypotheses of \cref{thm:k} are in general necessary and then showing that for many fundamental groups falling within the purview of \cref{thm:k}, knowledge of the stable isomorphism class of $\pi_2(M) $ does not suffice and the $k$-invariant is required.

\subsubsection*{Acknowledgements}

We are delighted to have the opportunity to thank Diarmuid Crowley, Jim Davis, Matthias Kreck, Markus Land, Wolfgang L\"{u}ck, Ian Hambleton, Henrik R\"uping, Mark Ullmann and Christoph Winges for many useful and interesting discussions on this work. We also thank the anonymous referee for several useful suggestions that helped improve the exposition.

The authors thank the Max Planck Institute for Mathematics and the Hausdorff Institute for Mathematics in Bonn for financial support and their excellent research environments.
The second author gratefully acknowledges an NSERC Discovery Grant, and was partially supported by EPSRC New Investigator grant EP/T028335/1 and EPSRC New Horizons grant EP/V04821X/1.

\section{\texorpdfstring{$\CP$}{CP2}-stable classification}\label{sec:stable classification}

Let $(\pi,w)$ be a 1-type, that is a finitely presented group $\pi$ together with a homomorphism $w\colon \pi \to \Z/2$. Let $\xi\colon BSO \times B\pi \to BO$ be a fibration that classifies the stable vector bundle obtained as the direct sum of the orientation double cover $BSO \to BO$ and the line bundle $L \colon B\pi \to BO(1)$ classified by~$w$. Let $\Omega_4(\pi,w) := \Omega_4(\xi)$ denote the bordism group of closed 4-manifolds $M$ equipped with a $\xi$-structure, namely a lift $\wt{\nu} \colon M \to BSO\times B\pi$ of the stable normal bundle $\nu \colon M\to BO$ along $\xi$, modulo cobordisms with the analogous structure extending the $\xi$-structure on the boundary.

\begin{lemma}\label{lem:bordism}
The map $\Omega_4(\pi,w) \to H_0(\pi;\Z^w) \times H_4(\pi;\Z^w)$ given by sending $[M,\wt \nu]$ to the pair
$\big(\sigma(M), (p_2\circ\wt\nu)_*([M])\big)$ is an isomorphism. Here if $M$ is orientable $(w=0)$,  $H_0(\pi;\Z^w) \cong \Z$ and $\sigma(M)\in\Z$ is the signature.  If $w$ is nontrivial, then $H_0(\pi;\Z^w) \cong \Z/2$ and $\sigma(M) \in \Z/2$ denotes the Euler characteristic mod~$2$.  We use $[M]\in H_4(M;\Z^w)\cong\Z$ for the fundamental class of $M$ determined by $\wt\nu \colon M \to BSO \times B\pi$.
\end{lemma}

\begin{proof}
By the Pontryagin-Thom construction, the bordism group $\Omega_4(\pi,w)$ is isomorphic to $\pi_5(MSO \wedge Th(L))$, where $Th(L)$ is the Thom space corresponding to the real line bundle $L$ and $MSO$ is the oriented Thom spectrum.  Shift perspective to think of $\pi_5(MSO \wedge Th(L))$ as the generalised reduced homology group $\pi_5(MSO \wedge -)$ of $Th(L)$, isomorphic to the group $\wt{\Omega}_5^{SO}(Th(L))$ of reduced $5$-dimensional oriented bordism over $Th(L)$. The Atiyah-Hirzebruch spectral sequence gives an exact sequence
\[
0 \ra H_1(Th(L);\Z) \ra \wt{\Omega}_5^{SO}(Th(L)) \ra H_5(Th(L);\Z) \ra 0
\]
because in the range $0 \leq i \leq 4$ we have that $\Omega_i^{SO} \cong \Z$ for $i=0,4$ and zero otherwise. Since $w_1(L)=w$, the Thom isomorphism theorem~\cite[Theorem~IV.5.7]{Rudyak}, \cite[Theorem~3.31]{Luck-basic-intro} shows that
\[H_1(Th(L);\Z) \cong H_0(\pi;\Z^w) \quad \text{ and } \quad H_5(Th(L);\Z)\cong H_4(\pi;\Z^w).\]
Using this the above short exact sequence translates to
\[
0 \ra H_0(\pi;\Z^w) \ra \Omega_4(\pi,w) \ra H_4(\pi;\Z^w) \ra 0.
\]
In the orientable case $w=0$, the group $H_0(\pi;\Z^w) \cong H_0(\pi;\Omega_4^{SO}) \cong \Omega^{SO}_4 \cong \Z$, while in the nonorientable case $w \neq 0$, we have $H_0(\pi;\Z^w) \cong H_0(\pi;(\Omega^{SO}_4)^w) \cong \Z \otimes_{\La} (\Omega^{SO}_4)^w \cong \Z/2$.  In both cases, the image of the inclusion $H_0(\pi;\Z^w) \to \Omega_4(\pi,w)$ is generated by $[\CP]$ with trivial map to $B\pi$.

To see that the short exact sequence splits, use that the signature is an additive invariant of oriented bordism, that the Euler characteristic mod 2 is an additive bordism invariant, and that both the signature and the Euler characteristic mod 2 of $\CP$ equal $1$.

That the Euler characteristic mod 2 is indeed a bordism invariant of closed $4$-manifolds can be seen as follows. It suffices to show that every $4$-manifold that bounds a $5$-manifold has even Euler characteristic. By Poincar\'e duality, a closed $5$-manifold has vanishing Euler characteristic. Now let $M$ be the boundary of a $5$-manifold $W$, and consider the Euler characteristic of the closed $5$-manifold $W \cup_M W$. We have $0=\chi(W\cup_M W)= 2\chi(W)-\chi(M)$, hence $\chi(M) =2\chi(W)$ is even.
\end{proof}

By surgery below the middle dimension, any bordism class can be represented by $(M,\wt\nu)$ where the second component of $\wt{\nu}$, namely $c:=p_2\circ\wt\nu$, is a $2$-equivalence, inducing an isomorphism $c_* \colon \pi_1(M) \toiso \pi$. Note also that the first component of $\wt\nu$ is an orientation of the bundle $\nu(M)\oplus c^*(L)$ over $M$, and hence all the circles required in our surgeries have trivial,  hence orientable, normal bundle.

\begin{thm}\label{thm:stable}
Let $(M_i, \wt\nu_i)$, for $i=1,2$,  be closed $4$-manifolds with the same $1$-type $(\pi,w)$ and assume that the resulting classifying maps $c_i \colon M_i\to B\pi$ are 2-equivalences for $i=1,2$. If $(M_1, \wt\nu_1)$ and $(M_2, \wt\nu_2)$ are $\xi$-cobordant then there is a $\CP$-stable diffeomorphism
\[
M_1\#^r \, \CP \#^{\bar{r}} \, \bCP \cong M_2\#^s \, \CP\#^{\bar{s}} \, \bCP, \quad \text{ for some } r,\bar{r},s, \bar{s} \in \N_0,
\]
 inducing the isomorphism $(c_2)_*^{-1}\circ (c_1)_*$ on fundamental groups and preserving the orientations on $\nu(M_i)\oplus (c_i)^*(L)$.
 \end{thm}

\begin{remark}
 In the orientable case $(w=0)$ we have $r-\bar r=s-\bar s$, since the signatures of $M_1$ and $M_2$ coincide.
 For non-orientable $M_i$ $(w\neq 0)$, as discussed in the introduction the connected sum operation is well-defined without choosing a local orientation, and there is no difference between connected sum with $\CP$ and with $\bCP$.
 As a consequence, when $w \neq 0$ we can write the conclusion with $\bar r=0=\bar s$. We then must have $r \equiv s \mod 2$, since the mod 2 Euler characteristics of $M_1$ and $M_2$ coincide.
\end{remark}

Before proving \cref{thm:stable}, we explain how \cref{theorem:Kreck} from the introduction follows from \cref{thm:stable}.

\begin{proof}[Proof of \cref{theorem:Kreck}]
Suppose that we are given two closed, connected $4$-manifolds $M_1, M_2$ and $2$-equivalences $c_i \colon M_i\to B\pi$ with $(c_i)^*(w)=w_1(M_i)$, for $i=1,2$, such that for some choices of fundamental classes, $(c_1)_*[M_1] = (c_2)_*[M_2]$ up to $\pm\Aut(\pi)$.  Change $c_1$ and the sign of $[M_1]$, or equivalently, the orientation of $\nu(M_1)\oplus (c_1)^*(L)$, so that the images of the fundamental classes coincide. If necessary, add copies of $\CP$ or $\bCP$ to $M_1$ until the signatures (or the Euler characteristics mod $2$, as appropriate) agree. By Lemma~\ref{lem:bordism}, the resulting $\xi$-manifolds $(M_1,\wt\nu_1)$ and $(M_2,\wt\nu_2)$, are $\xi$-bordant, and hence by \cref{thm:stable} they are $\CP$-stably diffeomorphic.

For the converse, for a 4-manifold $M$ define a map $f\colon M \# \CP \to M$ that crushes $\CP \sm \mathring{D}^4$ to a point.  The map $f$ is degree one.  The classifying map $c \colon M \# \CP \to B\pi$ factors through $f$ since any map $\CP \to B\pi$ is null homotopic.  Therefore $c_*[M\# \CP] = c_*[M] \in H_4(\pi;\Z^w)$.
\end{proof}

\begin{proof}[Proof of Theorem~\ref{thm:stable}]
Let $(W,\wt\nu)$ be a compact 5-dimensional $\xi$-bordism between the two $\xi$-manifolds $(M_1, \wt\nu_1)$ and $(M_2, \wt\nu_2)$.  By surgery below the middle dimension on~$W$, we can arrange that $p_2\circ\wt\nu \colon W \to B\pi$ is a $2$-equivalence and hence that both inclusions $M_i \hookrightarrow W$, $i=1,2$, are isomorphisms on fundamental groups. To make sure that the normal bundles to the circles we surger are trivial (so orientable), we use that $\wt\nu$ pulls back $w$ to the first Stiefel-Whitney class of $W$, and that we perform surgery on circles representing  elements of $\pi_1(W)$ that become null homotopic in $B\pi$.

Pick an ordered Morse function on $W$, together with a gradient-like vector field, and consider the resulting handle decomposition: $W$ is built from $M_1 \times [0,1]$ by attaching $k$-handles for $k=0,1,2,3,4,5$, in that order. The resulting upper boundary is $M_2$. Since $M_i$ and $W$ are connected, we can cancel the 0- and 5-handles.
Since the inclusions $M_i\hookrightarrow W$ induce epimorphisms on fundamental groups, we can also cancel the 1- and 4-handles. Both these handle cancelling manoeuvres are well-known and used in the first steps of the proof of the $s$-cobordism theorem e.g.~ \cite[Proposition~5.5.1]{Wall-differential-topology}.  No Whitney moves are required, so this handle cancelling also works in the 5-dimensional cobordism setting.
We are left with 2- and 3-handles only.

Next, injectivity of $\pi_1(M_1) \to \pi_1(W)$ shows that the $2$-handles are attached trivially to $M_1$, noting that homotopy implies isotopy for circles in a 4-manifold~\cite{Hudson-72}.   Similarly, the $3$-handles are attached trivially to $M_2$.
As a consequence, the middle level $M\subset W$ between the $2$- and the $3$-handles is diffeomorphic to both the outcome of 1-surgeries on $M_1$ along trivial circles and the outcome of 1-surgeries on $M_2$ on trivial circles. A 1-surgery on a trivial circle changes $M_i$ by connected sum with an oriented $S^2$-bundle over $S^2$.  There are two such bundles since $\pi_1(\Diff(S^2)) = \pi_1(O(3)) = \Z/2$. The twisted bundle occurs if and only if the twisted framing of the normal bundle to the trivial circle is used for the surgery.
One can prove, for example using Kirby's handle calculus, that after connected sum with $\CP$ both $S^2 \times S^2$ and $S^2 \wt{\times} S^2$ become diffeomorphic to $\CP \# \bCP \# \CP$ \cite[Corollaries~4.2~and~4.3]{Kirby-4-manifold-book}, \cite[p.~151]{Gompf-Stipsiz}.

The conclusions on the fundamental groups and relative orientations follow because these aspects are controlled through the $\xi$-structure of the bordism.
\end{proof}

\begin{remark}
  Even the null bordant class exhibits interesting behaviour.  A standard construction of a $4$-manifold with a given fundamental group $\pi$ and orientation character $w$ takes the boundary of some $5$-dimensional manifold thickening $N(K)$ of a 2-complex $K$ with $\pi_1(K) = \pi$ and $w_1(N(K)) = w$.  This is the same as doubling a suitable 4-dimensional thickening of $K$ along its boundary.  By construction $W_K:= \partial N(K)$ is null-bordant over $B\pi$ for any choice of $K$.  Thus by \cref{thm:stable} the $\CP$-stable diffeomorphism class of $W_K$ only depends on the 1-type $(\pi,w)$, and not on the precise choices of $K$ and $N(K)$.
\end{remark}

We end this section by outlining Kreck's original argument~\cite[Thm.~C]{surgeryandduality}. The normal 1-type of a manifold is determined by the data $(\pi,w_1,w_2)$. After adding one copy of $\CP$, the universal covering becomes non-spin and then the fibration $\xi \colon BSO \times B\pi\to BO$ is the normal 1-type. After adding more copies of $\CP$ to $M_1$ or $M_2$ until their signatures agree, our hypothesis gives a bordism between $M_1$ and $M_2$ over this normal 1-type. By subtraction of copies of $S^2 \times D^3$ tubed to the boundary, a cobordism can be improved to an $s$-cobordism, after allowing connected sums of the boundary with copies of $S^2 \times S^2$.  Therefore by the stable $s$-cobordism theorem~\cite{Quinn83}, the bordism class of a $4$-manifold in $\Omega_4(\xi)$ determines the diffeomorphism class up to further stabilisations with $S^2 \times S^2$. As remarked above, if necessary we may add one more $\CP$ to convert all the $S^2 \times S^2$ summands to connected sums of copies of $\CP$ and $\bCP$.

\section{Conventions} \label{sec:conventions}

Let $\pi$ be a group and write $\La:=\Z\pi$ for the group ring. For a homomorphism $w\colon \pi\to\Z/2$,
write $\La^w$ for the abelian group $\La$ considered as a $(\La,\La)$-bimodule, via the usual left action, and with the right action twisted with $w$, so that for $r \in \La$ and $g \in \pi$ we have $r \cdot g = (-1)^{w(g)}r g$. Note that $\La$ and $\La^w$ are isomorphic as left or right modules, but \emph{not} as bimodules.

Let $R$ be a ring with involution and let $N$ be an $(R,\La)$-bimodule.
We define another $(R,\La)$-bimodule $N^w := N \otimes_{\La} \La^w$. This is canonically isomorphic to the same left $R$-module $N$ with the right $\La$ action twisted with $w$.
We consider $\La$ as a ring with involution using the \emph{untwisted} involution determined by $g\mapsto g^{-1}$.

For a CW-complex $X$, or a manifold with a handle decomposition, write $\pi := \pi_1(X)$. We always assume that $X$ is connected and comes with a single 0-cell, respectively 0-handle.
The cellular or handle  chain complex $C_*(\wt{X}) \cong C_*(X;\La)$ consists of left $\La$-modules. Here we pick a base point in order to identify $C_*(\wt{X})$ with $C_*(X;\La)$.

Define the \emph{homology of $X$ with coefficients in $N$} as the left $R$-module
\[H_*(X;N) := H_*(N \otimes_{\La} C_*(X;\La)).\]
Define the \emph{cohomology of $X$ with coefficients in $N$} as the left $R$-module
\[H^*(X;N) := H^*\big(\Hom_{\La}(\ol{C_*(X;\La)},N)\big),\]
converting the chain complex into a right $\La$-module chain complex using involution on $\La$, taking $\Hom$ of right $\La$-modules, and using the left $R$-module structure of $N$ for the $R$-module structure of the outcome.

Given a chain complex $C_*$ over a ring with involution $R$, consisting of left $R$-modules, the cochain complex $C^* := \Hom_R(C_*,R)$ consists naturally of right modules.  Unless explicitly mentioned otherwise, we always convert such a cochain complex into left modules using the involution on $R$, that is $r\cdot c = c\ol{r}$.

We will consider closed manifolds to always be connected and smooth unless otherwise explicitly mentioned, and typically of dimension four.
For an $n$-dimensional closed manifold $M$ with a handle decomposition, write $M^\natural$ for the dual handle decomposition.   The handle chain complex $C_*(M;\La)$ satisfies $C_*(M^\natural;\La) \cong \La^w \otimes_{\La} C^{n-*}(M;\La)$. That is, the handle chain complex associated to the dual decomposition is equal to the cochain complex of the original decomposition defined using the twisted involution.  Let $d_i \colon C_i(M;\La) \to C_{i-1}(M;\La)$ be a differential in the chain complex.  For emphasis, we write \[d^i_w := \Id_{\La^w} \otimes d^i \colon C^{i-1} \to C^i\]
to indicate the differentials of the cochain complex obtained using the twisting, which coincide with the differentials of the dual handle decomposition.

A popular choice for coefficient module $N$ will be $\Z^w$, the abelian group $\Z$ considered as a $(\Z,\La)$-bimodule via the usual left action, and with the right action of $g \in \pi$ given by multiplication by $(-1)^{w(g)}$.
An $n$-dimensional closed manifold $M$, with fundamental group $\pi$ and orientation character $w \colon \pi \to \Z/2$, has a \emph{twisted fundamental class} $[M] \in H_n(M;\Z^w)$.

\begin{remark}\label{remark:choice-fund-class}
Observe that the orientation double cover $\wh{M}$ is canonically oriented.  Nevertheless in our context there is still a choice of fundamental class to be made. This arises from the fact that the identification $H_n(\wh{M};\Z) \cong H_n(M;\Z[\Z/2])$ requires a choice of base point. The orientation class of $\wh{M}$ maps to either $1-T$ or $T-1$ times the sum of the top dimensional handles/cells of~$M$.  Evaluating the generator $T \in \Z/2$ to $-1$ yields a homomorphism $H_n(M;\Z[\Z/2]) \to H_n(M;\Z^w)$ with the image of $[\wh{M}]$ equal to $\pm 2 [M]$.  We see that although the double cover is canonically oriented, the twisted fundamental class obtained by this procedure depends on a choice of base point, so there is a choice required.
\end{remark}

Twisted Poincar\'e duality says that taking the cap product with a fundamental class $[M] \in H_n(M;\Z^w)$ gives rise to an isomorphism
\[-\cap[M] \colon H^{n-r}(M;N^w) \ra H_r(M;N)\]
for any $r$ and any coefficient bimodule $N$.  Since $N^{ww} \cong N$, applying this to $N^w$ yields the other twisted Poincar\'e duality  isomorphism
\[-\cap[M] \colon H^{n-r}(M;N) \ra H_r(M;N^w).\]

\section{An extended Hopf sequence}\label{section:Hopf-sequence}

We present an exact sequence, extending the well-known Hopf sequence, for groups $\pi$ that satisfy $H^1(\pi;\La)=0$, where $\La := \Z\pi$.

Recall from \cref{sec:conventions} that an orientation character $w\colon \pi \to \Z/2$ endows $\Z^w:=\Z\otimes_\Lambda\Lambda^w$ with a  $(\Z,\La)$-bimodule structure.
The $\Z^w$-twisted homology of a space $X$ with $\pi_1(X)=\pi$ is defined as the homology of the chain complex $\Z^w \otimes_{\La} C_*(\wt{X})$.

In the upcoming theorem, write $\pi_2(M) ^w := \pi_2(M) \otimes_{\Lambda} \Lambda^w$,  where we consider $\pi_2(M) $ as a $\Lambda$-right module using the involution given by $g\mapsto g^{-1}$.  Then $\pi_2(M) ^w$ is a $(\Z,\La)$-bimodule, so we can use it as the coefficients in homology as in \cref{sec:conventions}.

\begin{theorem}
	\label{thm:hopfseq}
	Let $M$ be a closed 4-manifold with classifying map $c\colon M \to B\pi$ and orientation character $w \colon \pi \to \Z/2$.  If $H^1(\pi;\La) =0$
	there is an exact sequence
	\begin{align*}
	& H_4(M;\Z^w) \xrightarrow{\,c_*\,}  H_4(\pi;\Z^w)\xrightarrow{\,\partial\,}H_1(\pi;\pi_2(M) ^w) \\ \ra\, & H_3(M;\Z^w)\xrightarrow{\,c_*\,}  H_3(\pi;\Z^w)\xrightarrow{\,\partial\,} H_0(\pi;\pi_2(M) ^w) \\ \ra\, & H_2(M;\Z^w)\xrightarrow{\,c_*\,}H_2(\pi;\Z^w) \xrightarrow{\,\phantom{\partial}\,} 0.
	\end{align*}
	Moreover, the maps $\partial$ only depend on the 2-type $(\pi,w,\pi_2(M) ,k(M))$.
\end{theorem}
\begin{rem}
	When replacing $M$ by its Postnikov 2-type, the sequence in \cref{thm:hopfseq} remains exact even without the assumption that $H^1(\pi;\Z\pi)=0$ by \cite[Lemma~$8^{\mathrm{bis}}$.27]{mccleary}. The proof of \cref{thm:hopfseq} is essentially the same but we repeat it here for the reader's convenience.
\end{rem}

\begin{proof}
	The Leray-Serre spectral sequence applied to the fibration $\wt{M} \to M \to B\pi$ with homology theory $H_*(-;\Z^w)$ has second page \[E^2_{p,q} = H_p(B\pi;H_q(\wt{M};\Z^w)) = H_p(B\pi; H_q(\wt{M};\Z)\otimes_\Lambda\Lambda^w).\] Here we consider $H_q(\wt{M};\Z)$ as a right $\Lambda$-module and the fact that $\wt{M}$ is simply connected means that the $w$-twisting can be taken outside the homology. The spectral sequence converges to $H_{p+q}(M;\Z^w)$.
	
	First, $H_1(\wt{M};\Z) =0$, and \[H_3(\wt{M};\Z) \cong H_3(M;\La) \cong H^1(M;\La^w) \cong H^1(M;\La) \cong H^1(\pi;\La) = 0.\]  Thus the $q=1$ and $q=3$ rows of the $E^2$ page vanish.
	Since $H_0(\wt{M};\Z) \cong \Z$, the $q=0$ row coincides with the group homology $E^2_{p,0} = H_p(\pi;\Z^w)$.  We can write $H_2(\wt{M};\Z) \cong \pi_2(\wt{M})\cong \pi_2(M) $ by the Hurewicz theorem, and the long exact sequence in homotopy groups associated to the fibration above.  Therefore the $q=2$ row reads as \[E^2_{p,2} = H_p(\pi;\pi_2(M) ^w).\]
	
	Since the $q=1$ and $q=3$ lines vanish, the $d^2$ differentials with domains of degree $q \leq 2$ vanish, so we can turn to the $E^3$ page.  We have $d^3$ differentials
	\[
	d^3_{3,0} \colon H_3(\pi;\Z^w) \ra H_0(\pi;\pi_2(M) ^w),
	\]
	\[
	d^3_{4,0} \colon H_4(\pi;\Z^w) \ra H_1(\pi;\pi_2(M) ^w).
	\]
	It is now a standard procedure to obtain the long exact sequence from the spectral sequence, whose highlights we elucidate.
	On the 2-line, the terms on the $E^\infty$ page yield a short exact sequence
	\[0 \ra \coker (d^3_{3,0}) \ra H_2(M;\Z^w) \ra H_2(\pi;\Z^w) \ra 0.\]
	On the 3-line, similar considerations give rise to a short exact sequence
	\[0 \ra \coker d^3_{4,0} \ra H_3(M;\Z^w) \ra \ker d^3_{3,0} \ra 0.\]
	Finally the 4-line gives rise to a surjection
	\[H_4(M;\Z^w) \ra \ker d^3_{4,0} \ra 0.\]
	Splice these together to yield the desired long exact sequence.
	
	It remains to argue that the $d^3$ differentials in the Leray-Serre spectral sequence only depend on the $k$-invariant of $M$. To see this, consider the map of fibrations
	\[\xymatrix{ \wt M\ar[r]\ar[d]& M\ar[r]\ar[d]& K(\pi,1)\ar[d]^=\\
		K(\pi_2(M) ,2)\ar[r]&P_2(M)\ar[r]&K(\pi,1)}\]
	induced by the 3-equivalence from $M$ to its second Postnikov section $P_2(M)$. It induces a map of the spectral sequences, and since the map $\wt M\to K(\pi_2(M) ,2)$ is an isomorphism on homology in degrees $0,1$ and $2$, it follows that the two $d^3$ differentials in the long exact sequences can be identified. Therefore, they only depend on $P_2(M)$, or equivalently, on $(\pi_1(M) ,\pi_2(M) ,k(M))$.
\end{proof}

\begin{cor}
	\label{cor:hopfseq}
	Suppose that $H^1(\pi;\La) =0$ and that $M$ is a closed $4$-manifold with 1-type $(\pi,w)$. Then the subgroup generated by $c_*[M]\in H_4(\pi;\Z^w)$ only depends on the 2-type $(\pi,w,\pi_2(M) ,k(M))$.	
\end{cor}

\begin{proof}
	The subgroup generated by $c_*[M]$ is precisely the image of $c_*\colon H_4(M;\Z^w)\to H_4(\pi;\Z^w)$. By \cref{thm:hopfseq} the image of $c_*$ is the same as the kernel of the map $\partial\colon H_4(\pi;\Z^w)\to H_1(\pi;\pi_2(M) ^w)$. Since the latter only depends on the 2-type of $M$, so does the image of $c_*$.
\end{proof}

In particular, if $H_4(\pi;\Z^w)$ is torsion-free and $H^1(\pi;\La)=0$, then since the subgroup generated by $c_*[M]$ determines $c_*[M]$ up to sign, \cref{cor:hopfseq} implies that the 2-type of $M$ suffices to determine its $\CP$-stable diffeomorphism class.
Since a group with one end has $H^1(\pi;\La)$, this proves the special case of \cref{thm:k}~\eqref{thm:k:item2}, if in addition we assume that $H_4(\pi;\Z^w)$ is torsion-free. 
We will also make use of \cref{cor:hopfseq} to deduce \cref{thm:k}~\eqref{thm:k:item3}, but we postpone this discussion until the end of \cref{sec:k-ext}, so that we can collect the facts needed to prove \cref{thm:k} in one place.

\section{Computing fourth homology as an extension}
\label{sec:pri}
In this section we prove \cref{prop:mainpri2}, relating the $\CP$-stable classification to the stable isomorphism class of the second homotopy group as an extension.

\begin{definition}
\label{def:stableiso}
For a ring $R$, we say that two $R$-modules $P$ and $Q$ are \emph{stably isomorphic}, and write $P \scong Q$, if there exist non-negative integers $p$ and $q$ such that $P \oplus R^p \cong Q \oplus R^q$.
\end{definition}

First we will construct an isomorphism using a closed $4$-manifold $M$ with fundamental group $\pi$ and later recast it in terms of homological algebra to show that it is independent of $M$.
We will use the following description of the extension group.

\begin{rem}
	\label{rem:ext}
	Let $R$ be a ring and let $N$ and $L$ be $R$-modules. Recall that $\Ext_R^1(L,N)$ can be described as follows; see e.g.~\cite[Chapter~IV.7]{HS}. Choose a projective resolution
\[\cdots \ra P_2\xrightarrow{\,p_2\,} P_1\xrightarrow{\,p_1\,}P_0\ra L \ra 0\]
 of $L$ and consider the induced dual sequence
	\[
	0\ra \Hom_R(L,N)\ra \Hom_R(P_0,N)\xrightarrow{\,(p_1)^*}\Hom_R(P_1,N)\xrightarrow{\,(p_2)^*}\Hom_R(P_2,N).
	\]
	Then there is a natural isomorphism
\[\theta \colon \Ext_R^1(L,N) \to H_1(\Hom_R(P_*,N)).\]
This isomorphism is given as follows: for an extension $0\to N\to E\to L\to 0$, choose a lift $P_0\to L$ to a map $P_0\to E$ and note that the composition $P_1\xrightarrow{p_1}P_0\to E$ factors through a map $P_1\to N$. This element in $\Hom_R(P_1,N)$ is a $1$-cocycle and represents the image of the extension $E$ in $H^1(\Hom_R(P_*,N))$. Its cohomology class does not depend on the choice of lift.
\end{rem}

Let $\La = \Z\pi$, and let $(D_*,d_*^D)$ be a free resolution of $\Z$ as a $\La$-module, where $\pi$ acts trivially. Standard dimension shifting~\cite[Chapter~III.7]{brown} gives an isomorphism $H_4(\pi;\Z^w)\xrightarrow{\cong} H_1(\pi;\ker d_2^{D,w})$. Here the superscript $w$ denotes twisting the coefficients using $w\colon \pi\to \Z/2$ as described in \cref{sec:conventions}.

Given a closed $4$-manifold $M$ with fundamental group $\pi$ and orientation character $w$, we have isomorphisms
\[H_1(\pi;\ker d_2^{D,w})\xleftarrow{c_*, \cong}H_1(M;\ker d_2^{D,w})\xleftarrow{PD,\cong}H^3(M;\ker d_2^D),\]
where $c\colon M\to B\pi$ denotes the classifying map and $PD$ denotes Poincar\'e duality. Let $P_2(M)$ be the Postnikov 2-type of $M$.

\begin{lemma}\label{lemma:some-induced-map}
 The induced map $H^3(P_2(M);\ker d_2^D)\to H^3(M;\ker d_2^D)$ is an isomorphism.
\end{lemma}

\begin{proof}
 Since
 the map $M\to P_2(M)$ is $3$-connected, the induced map $H^2(P_2(M);A) \xrightarrow{\cong} H^2(M;A)$ is an isomorphism and the map $H^3(P_2(M);A) \hookrightarrow H_3(M;A)$ is injective for any coefficient system $A$. Using this and $H^3(M;D_2)\cong H_1(M;D_2)=0$, we obtain the following diagram associated with the short exact sequence $0 \to \ker d_2^D\to D_2\to \im d_2^D \to 0$.
\[
\xymatrix{
H^2(P_2(M);\im d_2^D)\ar[d]^\cong\ar[r]&H^3(P_2(M);\ker d_2^D)\ar[r]\ar@{^(->}[d]&H^3(P_2(M);D_2)\ar@{^(->}[d]\\
H^2(M;\im d_2^D)\ar@{->>}[r]&H^3(M;\ker d_2^D)\ar[r]&0\\
}
\]
It follows from commutativity of the left square that the middle vertical map $H^3(P_2(M);\ker d_2^D)\to H^3(M;\ker d_2^D)$ is surjective, and thus an isomorphism.
\end{proof}

By the Hurewicz theorem, $0=\pi_3(\wt{P_2(M)})\to H_3(\wt{P_2(M)};\Z)\cong H_3(P_2(M);\La)$ is onto, so $H_3(P_2(M);\La)=0$. Hence
\[C_4^{P_2(M)}\xrightarrow{d_4^{P_2(M)}} C_3^{P_2(M)}\xrightarrow{d_3^{P_2(M)}} C_2^{P_2(M)}\to \coker d_3^{P_2(M)} \to 0\]
is exact, where $(C_*^{P_2(M)},d_*^{P_2(M)})$ denotes the cellular $\La$-chain complex of $P_2(M)$.
This implies that
\begin{equation}\label{eqn:map-from-theta}
  H^3(P_2(M);\ker d_2^D) \cong \Ext_{\La}^1(\coker d_3^{P_2(M)},\ker d_2^D),
\end{equation}
 using the description from \cref{rem:ext}. As we can obtain $P_2(M)$ from $M$ by attaching cells of dimension $4$ and higher, we have $d_3^{P_2(M)}=d_3^M$.

We obtain the chain of isomorphisms
\begin{align}
\label{eq:theiso}
H_4(\pi;\Z^w)&\xrightarrow{\cong}H_1(\pi;\ker d_2^{D,w})\xleftarrow{c_*, \cong}H_1(M;\ker d_2^{D,w})\xleftarrow{PD;\cong}H^3(M;\ker d_2^D) \\
&\xleftarrow{\cong}H^3(P_2(M);\ker d_2^D)\cong \Ext_{\La}^1(\coker d_3^M,\ker d_2^D). \nonumber
\end{align}

We will now recast this isomorphism using homological algebra. Let $(D_*,d_*^D)$ and $(D'_*,d_*^{D'})$ be free resolutions of $\Z$ as a $\La$-module with $D'_i$ finitely generated for $i=0,1,2$. Such a resolution exists since the fundamental group of a closed manifold is finitely presented. 
Using the resolution $D'_*$, we have
\begin{align*}
  H_4(\pi;\Z^w)& \cong H_1(\pi;\ker d_2^{D,w})\cong H_1(D_*'\otimes_\La \ker d_2^{D,w})\cong H_1(D^{\prime,w}_*\otimes_{\La}\ker d_2^D) \\ &\cong H_1(\Hom_{\La}(D^{\prime,*}_w,\ker d_2^{D})).
\end{align*}
The last isomorphism uses that $D_i'$ is finitely generated for $i=0,1,2$.

Let $P_* \to \coker d^{D^\prime,2}_{w}$ be a free resolution.
As in \cref{rem:ext}, we have an isomorphism
\[\theta \colon \Ext_{\La}^1(\coker d^{D^\prime,2}_{w},\ker d_2^{D}) \xrightarrow{\cong} H^1(\Hom_\La(P_*,\ker d_2^{D})).\]
By the fundamental lemma of homological algebra, there is a chain map $D_w^{\prime,2-*} \to P_*$, unique up to chain homotopy, as follows:
\[\xymatrix @R-0.3cm {0 \ar[r] & D_w^{\prime,0} \ar[r]^{d^{D',1}_w} \ar[d] & D_w^{\prime,1} \ar[r]^{d^{D',2}_w} \ar[d] & D_w^{\prime,2} \ar@{->>}[r] \ar[d] & \coker d^{D^\prime,2}_{w} \\
\cdots \ar[r] & P_2 \ar[r] & P_1 \ar[r] & P_0 \ar@{->>}[ur] & &
}\]
 This induces a homomorphism
\[\Xi \colon H^1(\Hom_\La(P_*,\ker d_2^{D})) \to H^1(\Hom_\La(D^{\prime,2-*}_w,\ker d_2^D))\cong H_1(\Hom_{\La}(D^{\prime,*}_w,\ker d_2^{D})),\] so composing this with $\theta$ we obtain a homomorphism $\Ext_{\La}^1(\coker d^{D^\prime,2}_{w},\ker d_2^D)\to H_1(\Hom_{\La}(D^{\prime,*}_w,\ker d_2^{D}))$.

\begin{lemma}\label{lem:ext-H1}
  The homomorphism \[\Xi \colon H^1(\Hom_\La(P_*,\ker d_2^{D})) \to H_1(\Hom_{\La}(D^{\prime,*}_w,\ker d_2^{D}))\]
  is an isomorphism.
\end{lemma}

If $D_w^{\prime,0} \xrightarrow{} D_w^{\prime,1} \xrightarrow{}  D_w^{\prime,2}$ happens to be exact, in other words if $H^1(\pi;\La)=0$, then the lemma is straightforward because $D^{\prime,2-*}$ is the start of another free resolution of $\coker d^{D^\prime,2}_{w}$. In general this sequence is not exact.
We postpone the proof of the lemma in order to first discuss some of its consequences.

Consider the composition
\begin{align}
\label{eq:theiso2}
\Psi(D_*,D_*') \colon  H_4(\pi;\Z^w) &
\cong   H_1(\Hom_{\La}(D^{\prime,*}_w,\ker d_2^{D})) \nonumber \\
& \cong  H^1(\Hom_\La(P_*,\ker d_2^{D})) \nonumber \\
 &\cong  \Ext_{\La}^1(\coker d^{D^\prime,2}_{w},\ker d_2^D),
\end{align}
using the inverse of the isomorphism $\Xi$ from \cref{lem:ext-H1}.  The composition $\Psi(D_*,D_*')$ is natural in $D_*$ and $D_*'$.
This implies the following lemma.

\begin{lemma}
	\label{lem:naturality-ext}
	For all chain homotopy equivalences $D_*\xrightarrow{f} E_*$ and $D_*'\xrightarrow{f'} E_*'$ of free resolutions, with $D'_*,E'_*$ finitely generated for $i=0,1,2$, the induced map $$\Ext_{\La}^1((f')^*,f_*)\colon \Ext_{\La}^1(\coker d^{D^\prime,2}_{w},\ker d_2^D)\to \Ext_{\La}^1(\coker d^{E^\prime,2}_{w},\ker d_2^E)$$ is an isomorphism and $\Psi(E_*,E_*')=\Ext_{\La}^1((f')^*,f_*)\circ \Psi(D_*,D_*')$.
\end{lemma}

\cref{lem:naturality-ext} together with \eqref{eq:theiso2} implies \cref{prop:mainpri1}.
We will now explain how \eqref{eq:theiso2} can be identified with \eqref{eq:theiso}.
Let $M$ be a closed, smooth $4$-manifold with a 2-equivalence $c \colon M \to B\pi$ and an element $w \in H^1(\pi;\Z/2)$ such that $c^*(w) = w_1(M)$. Consider a handle decomposition of $M$ with a single $0$-handle and a single $4$-handle. Let $(C^M_*,d_*^M)$ be the corresponding $\La$-chain complex. Let $(C^\natural_*,\delta_*)$ denote the $\La$-chain complex for the dual handle decomposition. Observe that under the canonical identification of $C^{\natural}_i = \La^w\otimes_{\La}C^{4-i}:=C^{4-i}_w$, we have $\delta_i = d^{5-i}_{M,w}$ and $\delta^i_w = d^M_{5-i}$. Hence picking $D_*$ and $D'_*$ so that $D_i:= C_i^M$ and $D'_i:= C^\natural_w$ for $i=0,1,2$, the isomorphism $\Psi(D_*,D'_*)$ can be identified with \eqref{eq:theiso}.

Now we begin the promised postponed proof of \cref{lem:ext-H1}.

\begin{lemma}
	\label{lem:kerexact}
	Let $D_2\xrightarrow{d_2} D_1\xrightarrow{d_1} D_0$ be an exact sequence of finitely generated projective $\La$-modules, and let $f\colon L\to D^1$ be a $\La$-homomorphism such that $L\xrightarrow{f} D^1\xrightarrow{d^2}D^2$ is exact.  Then
	\[D_2\xrightarrow{\,d_2\,}D_1\xrightarrow{\,f^*}\Hom_{\La}(L,\La)\]
	is also exact.
\end{lemma}

\begin{proof}
	First note that $f^* \circ d_2 = (d^2 \circ f)^*=0$, and thus $\im d_2\subseteq \ker f^*$. We need to show that $\ker f^* \subseteq \im d^2$. Dualise $D_2\xrightarrow{d_2} D_1\xrightarrow{d_1} D_0$ we have that $d^2 \circ d^1 =0$. Using this and the hypothesis on $f$ we have a commutative diagram, showing that both $d^1$ and $f$ factor through $\ker d^2$:
	\[
	\xymatrix @R-0.2cm {D^0\ar[dr]^-{d^1} \ar[d] & & \\ \ker d^2 \enskip \ar@{^{(}->}[r] &  D^1 \ar[r]^-{d^2} & D^2.   \\
		L \ar@{->>}[u] \ar[ur]_-{f} && }
	\]
	Dualise, and identify $D_0, D_1$ and $D_2$ with their double duals, to obtain the diagram:
	\[
	\xymatrix @R-0.2cm {
& & D_0 \\
D_2\ar[r]^-{d_2}& D_1\ar[ur]^-{d_1}\ar[dr]_-{f^*} \ar[r] & \Hom_{\La}(\ker d^2,\La) \ar[u] \ar@{^{(}->}[d]\\
& &  \Hom_{\La}(L,\La).
}\]
	Since $\Hom_{\La}(-,\La)$ is left exact, $\Hom_{\La}(\ker d^2,\La)\to \Hom_{\La}(L,\La)$ is injective as shown.
	It now follows from the diagram that every element in the kernel of $f^* \colon D_1\to \Hom_{\La}(L,\La)$ is also in the kernel of $d_1$, and hence by exactness of $D_*$ lies in the image of $d_2$. This shows that $\ker f^* \subseteq \im d^2$ as desired, which completes the proof of the lemma.
\end{proof}

The next lemma implies \cref{lem:ext-H1} by taking $N:=\ker d_2^D$ and $E:= D'_w$.

\begin{lemma}\label{lem:ext}
	Let $E_2\xrightarrow{d_2} E_1\xrightarrow{d_1} E_0$ be an exact sequence of finitely generated projective $\La$-modules. Let $N$ be a submodule of a finitely generated free $\La$-module $F$. Let $P_* \to \coker d^2$ be a projective resolution. Then
\[\Xi \colon H^1(\Hom_\La(P_*,N)) \to H^1(\Hom_\La(E^{2-*},N))\cong H_1(\Hom_{\La}(E^{*},N))\]
is an isomorphism.
\end{lemma}

\begin{proof}
By the fundamental lemma of homological algebra, we can assume that $P_1 = E^1 \xrightarrow{d^2} P_0 = E^2$, and obtain a chain map $E^{2-*} \to P_*$, giving a commutative diagram
\[\xymatrix{ & E^0 \ar[r]^{d^{1}} \ar[d]_{\alpha} & E^{1} \ar[r]^{d^{2}} \ar@{=}[d] & E^{2} \ar@{->>}[r] \ar@{=}[d] & \coker d^{2} \\
\cdots \ar[r] & P_2 \ar[r] \ar[ur]^{\beta} & P_1 \ar[r] & P_0 \ar@{->>}[ur] & &
}\]
where $\beta$ is defined by the composition $P_2 \to P_1 = E^1$.
We will consider three cases: (i) $N=\Z\pi$; (ii) $N$ is a finitely generated free module; and (iii) $N$ is a submodule of a finitely generated free module~$F$.

For case (i), we identify $\Hom_{\La}(E^i,\Z\pi)$ with $E_i$, from which it follows that $H_1(\Hom_{\La}(E^*,\Z\pi))=0$. By \cref{lem:kerexact} with $D_i=E_i$ for $i=0,1,2$ and $L=P_2$, we deduce that $H^1(\Hom_{\La}(P_*,\Z\pi))$ is also trivial.

For case (ii), using additivity of both sides under direct sum, the statement follows immediately from case (i).

Finally we consider case (iii). 
It suffices to show that the kernel of $\beta_N^*$, shown in the next diagram, agrees with the kernel of $d^{1,*}_N \colon \Hom_{\La}(E^1,N)\to \Hom_\La(E^0,N)$.
Consider the diagram whose rows come from applying $\Hom_{\La}(-,M)$ to the factorisation $d^2 = \beta \circ \alpha$ for $M=N,F$:
\[\xymatrix @C-0.9cm {
	d^{1,*}_N \colon & \Hom_\La(E^1,N)\ar@{^(->}[d]\ar[rrrrrr]^-{\beta_N^*} &&&&&& \Hom_\La(P_2,N)\ar@{^(->}[d]\ar[rrrrrr]^-{\alpha_N^*} &&&&&& \Hom_\La(E^0,N)\ar@{^(->}[d]\\
	d^{1,*}_F \colon & \Hom_\La(E^1,F)\ar[rrrrrr]^-{\beta_F^*} &&&&&&\Hom_\La(P_2,F)\ar[rrrrrr]^-{\alpha_F^*} &&&&&&\Hom_\La(E^0,F).
}\]
By case (ii), in the bottom row the kernels of $d^{1,*}_F \colon \Hom_\La(E^1,F)\to\Hom_\La(E^0,F)$ and $\beta_F^* \colon \Hom_\La(E^1,F)\to\Hom_\La(P_2,F)$ both equal the image of $\Hom_{\La}(E^2,F)$, and in particular are equal. Now it follows from an easy diagram chase, using injectivity of the middle vertical map, that the kernels of $\beta_N^*$ and $d^{1,*}_N$ also agree, as desired.	This completes the proof of case (iii) and therefore of \cref{lem:ext}.
\end{proof}

	In order to give a uniform treatment in the following description of iterated Bockstein homomorphisms, define $d_0^M := \varepsilon \colon C_0^M \to C^M_{-1} := \Z$ and write $d_{-1}^M \colon \Z \to 0$.
	Then \[C_3^M \xrightarrow{d_3^M} C_2^M \xrightarrow{d_2^M} C_1^M \xrightarrow{d_1^M} C_0^M \xrightarrow{d_0^M} C_{-1}^M \xrightarrow{d_{-1}^M} 0.\]
	is exact at $C_i^M$ for $i \leq 1$.
	For every $i$ we have a short exact sequence
	\[0 \to \ker d_i^M \to C_i^M \xrightarrow{d_i^M} \im d_i^M \to 0,\]
	which for $i =0,1,2$ by exactness yields a short exact sequence
	\[0 \to \ker d_i^M \to C_i^M \xrightarrow{d_i^M}  \ker d_{i-1}^M \to 0.\]
	
	The isomorphism $H_4(\pi;\Z^w)\to H_1(\pi;\ker d^{M,w}_2)$ from dimension shifting is given by iterating the Bockstein homomorphism $H_{i+1}(\pi;\ker d_{2-i}^{M,w})\to H_i(\pi;\ker d_{3-i}^{M,w})$ for the above short exact sequences twisted by $w$. In order to trace the effect of this isomorphism in \cref{thm:implies-prim-theorems} below, we have to understand the iterated Bockstein in the following situation.

\begin{lemma}\label{lem:tracing-dim-shifting}
The iterated Bockstein $\beta^3\colon H^0(M;\Z) \to H^3(M;\ker d_2^M)$ for the above short exact sequences sends the class of the augmentation $\varepsilon \colon C^M_0\cong \La\to \Z$ to the class of $d^M_3 \colon C_3^M \to \ker d_2^M$.
\end{lemma}

\begin{proof}
Let $\beta \colon H^i(M;\ker d_{i-1}^M) \to H^{i+1}(M;\ker d_i^M)$ be the Bockstein homomorphism.
We claim that $\beta([d_i^M])= [d_{i+1}^M]$. Then it will follow using this for $i=0,1,2$ that $[\varepsilon] = [d_0^M] \in H^0(M;\ker d_{-1}^M) = H^0(M;\Z)$ is sent
to $[d_3^M] \in H^3(M;\ker d_2^M)$, as desired.
The claim follows from studying the definition of the Bockstein connecting homomorphism via the following diagram.
\[\xymatrix{ & \Hom_{\La}(C_i^M,C_i^M) \ar[r]^-{d_i^M \circ -} \ar[d]^-{-\circ d_{i+1}^M } & \Hom_{\La}(C_i^M,\ker d_{i-1}^M) \\
\Hom_{\La}(C_{i+1}^M,\ker d_{i}^M) \ar[r] & \Hom_{\La}(C_{i+1}^M,C_i^M) & }\]
Here $d_i^M \in \Hom_{\La}(C_i^M,\ker d_{i-1}^M)$ lifts to $\Id_{C_i^M} \in \Hom_{\La}(C_i^M,C_i^M)$, which is sent to $d^M_{i+1} \in \Hom_{\La}(C_{i+1}^M,\ker d_{i}^M)$, as claimed.
\end{proof}

For the proof of \cref{prop:mainpri2}, we now want to understand $\pi_2(M)$ as an extension class.
As before, let $M$ be a closed, smooth $4$-manifold with a 2-equivalence $c \colon M \to B\pi$ and an element $w \in H^1(\pi;\Z/2)$ such that $c^*(w) = w_1(M)$. Consider a handle decomposition of $M$ with a single $0$-handle and a single $4$-handle. Let $(C^M_*,d_*^M)$ be the corresponding $\La$-chain complex.
For a finite, connected 2-complex $K$ with fundamental group $\pi$, Hambleton showed \cite[Theorem 4.2]{hambleton-gokova} that $\pi_2(M) $ is stably isomorphic as a $\La$-module to an extension $E$ of the form
\[
0\ra H_2(K;\La)\ra E \ra H^{2}(K;\La^w)\ra 0.
\]
Note that in \cite[Theorem 4.2]{hambleton-gokova} only the oriented case was considered.
To identify the equivalence class of this extension with $c_*[M]\in H_4(\pi;\Z^w)$, which is the goal of this section, we need the following version of this theorem.

Let $i\colon \ker d_2^M\hookrightarrow C_2^M$ and $i'\colon \ker d_2^M/\im d_3^M\hookrightarrow \coker d_3^M$ denote the inclusions and let $p\colon \ker d_2^M \twoheadrightarrow \ker d^M_2/\im d^M_3$ and $p'\colon C^M_2\twoheadrightarrow \coker d^M_3$ denote the projections.

\begin{proposition}\label{thm:ext}
There is a short exact sequence
\[0\ra \ker d^M_2\xrightarrow{(i,-p)^T} C^M_2\oplus H_2(C^M_*)\xrightarrow{(p',i')} \coker d^M_3 \ra 0.\]
\end{proposition}

\begin{proof}
Consider the diagram
\[\xymatrix{
\ker d_2^M\ar[r]^-i\ar[d]^p&C^M_2\ar[d]^{p'}\\
H_2(C^M_*) =\ker d^M_2/\im d^M_3\ar[r]^-{i'}&\coker d^M_3.
}\]
It is straightforward to check that this square is a pullback as well as a push out.
Therefore, we obtain the claimed short exact sequence.
\end{proof}

We explain why \cref{thm:ext} coincides with \cite[Theorem 4.2]{hambleton-gokova}. 
In that reference $\coker d^M_3$ is replaced with $H^2(K;\La^w)$, and $\ker d^M_2$ is replaced with $H_2(K;\La)$, where~$K$ is a finite 2-complex with $\pi_1(K)=\pi$.

To see why one can make these replacements, we need the following lemma.
\begin{lemma}
	\label{lem:stable}
	Let $K_1, K_2$ be finite $2$-complexes with 2-equivalences $K_i\to B\pi$ for $i=1,2$. Then there exist $p,q\in \N_0$ such that $K_1\vee \bigvee_{i=1}^p S^2$ and $K_2\vee\bigvee_{i=1}^q S^2$ are homotopy equivalent over $B\pi$. In particular $H_2(K_1;\La)$ and $H_2(K_2;\La)$ are stably isomorphic $\La$-modules and so are $H^2(K_1;\La^w)$ and $H^2(K_2;\La^w)$.
\end{lemma}

\begin{proof}
	After collapsing maximal trees in the 1-skeletons of $K_i$, we can assume that both $K_i$ have a unique $0$-cell.
	The lemma follows from the existence of Tietze transformations that relate the resulting presentations of the group~$\pi$, by realising the sequence of  transformations on the presentation by cellular expansions and collapses. See for example \cite[(40)]{lms197}.
\end{proof}

It follows that $\ker d^M_2=H_2(M^{(2)};\La)$ is stably isomorphic to $H_2(K;\La)$.
Let $M^{\natural}$ be the manifold~$M$ endowed with the dual handle decomposition, and let $(C^{\natural}_*, \delta_*)$ be its $\La$-module chain complex as before.
Fix a choice of (twisted) fundamental class $[M]$ in $H_4(M;\Z^w)$.
Use this choice to identify $\ker \delta_2 = \ker d^3_{M,w}$ and $\coker d^M_3 = \coker \delta^2_w$.
It then follows from Lemma~\ref{lem:stable} that
\[
 \ker d^M_2 = H_2(M^{(2)};\La) \scong H_2((M^{\natural})^{(2)};\La) = \ker \delta_2 = \ker d^3_{M,w},
 \]
 \[
 \coker d^M_3 = \coker \delta^2_w = H^2((M^{\natural})^{(2)};\La^w) \scong H^2(M^{(2)};\La^w) = \coker d^2_{M,w}.
 \]
 We can now prove the following proposition, which is the same as \cref{prop:mainpri2}. Here we use the same choice of $[M]$ that we just fixed.

\begin{proposition}\label{thm:implies-prim-theorems}
Let $M$ be a closed 4-manifold with a $2$-equivalence $c\colon M\to B\pi$ and let $(C^M_*,d^M_*)$ be the chain complex from a handle decomposition of~$M$. Let $D_*$ be a free resolution of $\Z$ with $D_i=C_i^M$ for $i=0,1,2$ and let $D'_*$ be a free resolution of $\Z$ with $D'_i=C^\natural_i$ for $i=0,1,2$.  Then the isomorphism $\Psi(D_*,D'_*)$ takes $c_*[M] \in H_4(\pi;\Z^w)$ to the equivalence class of the extension from \cref{thm:ext}:
\[
0\ra \ker d^M_2 \ra C^M_2\oplus H_2(C^M_*) \ra \coker d^M_3\ra 0.
\]
\end{proposition}

\begin{proof}
	We use the identification of $\Psi(D_*,D'_*)$ with \eqref{eq:theiso}. Using the description of the extension group from \cref{rem:ext}, the extension from \cref{thm:ext} in $\Ext_{\La}^1(\coker d_3^M,\ker d_2^M)$ is represented by $d_3^M\in \Hom_\La(C_3^M,\ker d_2^M)$. It remains to show that \eqref{eq:theiso} sends $c_*[M]$ to the class $[d_3^M] \in \Ext_{\La}^1(\coker d_3^M,\ker d_2^M)$.

Consider the diagram
\[\xymatrix{
H_4(M;\Z^w)\ar[d]^{c_*}& H^0(M;\Z)\ar[r]^-{\beta^3}\ar[l]^-{PD}_-\cong &H^3(M;\ker d_2^M)\ar[d]^\cong_-{PD}\\
H_4(\pi;\Z^w)\ar[r]^-\cong& H_1(\pi;\ker d_2^w)\ar[r]^-\cong&H_1(M;\ker d_2^{M,w}).
}\]
The square commutes by naturality of Poincar\'e duality in the coefficients. Under Poincar\'e duality the fundamental class $[M]\in H_4(M;\Z^w)$ is mapped to the class in $H^0(M;\Z)$ represented by the augmentation $\varepsilon\colon C^M_0\cong \La\to \Z$.  By \cref{lem:tracing-dim-shifting},  
$\beta^3$ maps $[\varepsilon]$ to $[d_3^M] \in H^3(M;\ker d_2^M)$. So $[M]$ is sent by the top route to $[d_3^M]$.  On the other hand the bottom-then-up route $H_4(\pi;\Z^w) \to H^3(M;\ker d_2^M)$ is the first three isomorphisms of \eqref{eq:theiso}.  So by commutativity these send $c_*[M]$ to $[d_3^M]$.  Finally, since $P_2(M)$ is obtained from $M$ by attaching cells of dimension $4$ and higher, $[d_3^M]\in H^3(P_2(M);\ker d_2^M)$ is a pre-image of $[d_3^M]\in H^3(M;\ker d_2^M)$. It follows that $c_*[M]$ is mapped under \eqref{eq:theiso} to the extension represented by $[d_3^M]$, as desired.
\end{proof}

\section{Detecting the extension class}
\label{sec:k-ext1}

This section is entirely algebraic.
Let $\hCh_2(\pi)$ denote the category whose objects are free $\La$-chain complexes $C_*$ concentrated in non-negative degrees such that $H_n(C_*)=0$ for $n\neq 0,2$, together with a fixed identification $H_0(C_*)=\Z$, and whose morphisms are chain maps that induce the identity on $H_0$, considered up to chain homotopy.
Let $\La^n[2]$ denote the chain complex given by the based free $\La$-module $\La^n$ concentrated in degree $2$.
We call two chain complexes $C_*,C'_*$ stably isomorphic if there exists $p$, $q$ such that $C_*\oplus \La^p[2]$ and $C'_*\oplus\La^{q}[2]$ are isomorphic in $\hCh_2(\pi)$, meaning that they are chain homotopy equivalent.  We denote the set of stable isomorphisms classes by $\sCh_2(\pi)$.

For the rest of this section fix a free resolution $(C_*,d_*)$ of $\Z$ as a $\La$-module with $C_*$ finitely generated for $*=0,1,2$. Also fix a complex
\[\cdots \to B_4\xrightarrow{\,b_4\,} B_3\xrightarrow{\,b_3\,} B_2\]
of free $\La$-modules that is exact at $B_n$ for $n\geq 3$, with $B_2$ and $B_3$ finitely generated.
Furthermore, assume that the dual complex $B^2\xrightarrow{b^3}B^3\xrightarrow{b^4}B^4 \to \cdots$ is exact at $B^3$, or equivalently that $\Ext_{\La}^1(\coker b_3,\La)=0$.

Next we define a map $\Theta \colon \Ext_{\La}^1(\coker b_3,\ker d_2) \to \sCh_2(\pi).$  
Recall from \cref{rem:ext} that any element of $\Ext_{\La}^1(\coker b_3,\ker d_2)$ can be represented by a map $f\colon B_3\to \ker d_2$ with $f\circ b_4=0$. For such a map, let $C_f$ be the $\La$-chain complex
\[\xymatrix{\cdots\ar[r]&B_4\ar[r]^-{b_4}&
	B_3\ar[r]^-{(f,b_3)^T}&
	C_2\ar[r]^-{(d_2,0)}\oplus B_2&
	C_1\ar[r]^-{d_1}&
	C_0.}\]
Note that $H_*(C_f)=0$ for $*\neq 0,2$ and that $H_0(C_f)=\Z$, thus $C_f$ is an element in $\sCh_2(\pi)$.
This is our candidate for $\Theta([f])$. We need to check that it is well-defined.

\begin{rem}
	As in \cref{thm:ext}, there is a short exact sequence
	\begin{equation}
		\label{eq:ex}
		0\to \ker d_2\oplus B_2\to H_2(C_f)\oplus C_2\oplus B_2\to \coker (f,b_3)^T \to 0.
	\end{equation}
	In the proof of \cref{thm:injective} we show that $\coker (f,b_3)^T$ is isomorphic to $C_2\oplus \coker b_3$. Thus stably $H_2(C_f)$ represents an element of $\Ext_\La^1(\coker b_3,\ker d_2)$. Moreover, using the free resolution
\[B_4\to B_3\xrightarrow{(0,b_3)^T}C_2\oplus B_2\to C_2\oplus \coker b_3,\]
it can be shown that \eqref{eq:ex}, considered as an element of $\Ext_\La^1(C_2\oplus \coker b_3,\ker d_2\oplus B_2)$, is represented by $(f,b_3)^T\colon B_3\to \ker d_2\oplus B_2$,  or equivalently by $(f,0)^T\colon B_3\to \ker d_2\oplus B_2$.  We will not make use of this fact.
\end{rem}

\begin{lemma}\label{lemma:Theta-well-defined}
	Suppose that two maps $f,g\colon B_3\to \ker d_2$ differ by a coboundary $B_3 \xrightarrow{B_3} B_2 \xrightarrow{h} \ker d_2$, so represent the same extension class in $\Ext_{\La}^1(\coker b_3,\ker d_2)$.  Then $[C_f]=[C_g]\in \sCh_2(\pi)$.
\end{lemma}

\begin{proof}
	Two maps $f,g\colon B_3\to \ker d_2$ represent the same extension class if and only if $f-g=h\circ b_3$, for some $h \colon B_2 \to \ker d_2 \subseteq C_2$. We then have the following chain isomorphism.
	\[\xymatrix @C+1cm{
		\cdots\ar[r]&B_3\ar[r]^-{(f,b_3)^T} \ar[d]^{\id} &  C_2\oplus B_2\ar[r]^-{(d_2, 0)} \ar[d]^-{\left(\begin{smallmatrix}\id&-h\\0&\id\end{smallmatrix}\right)} & C_1\ar[r]^{d_1}\ar[d]^{\id} &	 C_0\ar[d]^{\id}
		\\
		\cdots\ar[r]&B_3\ar[r]^-{(g,b_3)^T} & C_2\oplus B_2  \ar[r]^-{(d_2,0)}  & C_1 \ar[r]^-{d_1} & C_0
	}\]
In particular, $[C_f]=[C_g]\in \sCh_2(\pi)$.
\end{proof}

Thus we obtain a well-defined map as follows:
\begin{equation}\label{eqn:Theta}
\begin{array}{rcl} \Theta\colon \Ext_{\La}^1(\coker b_3,\ker d_2) & \ra & \sCh_2(\pi) \\ {[f]} & \mapsto & [C_f]. \end{array}
\end{equation}

Our next aim is to find a suitable quotient of the domain that converts this map into an injection.
The inclusion $\coker b_3\to \coker b_3\oplus \La$ induces an isomorphism (cf. \cref{lem:naturality-ext})
\[\Ext_{\La}^1(\coker b_3,\ker d_2)\ra \Ext_{\La}^1(\coker b_3\oplus \La,\ker d_2),\]
and hence any automorphism of $\coker b_3\oplus \La^n$ acts on $\Ext_{\La}^1(\coker b_3,\ker d_2)$.

\begin{lemma}\label{lem:theta-well-def-cokernel}
	The map $\Theta$ from~\eqref{eqn:Theta} is invariant under the action of $\alpha \in \Aut(\coker b_3\oplus\La^n)$ on $\Ext_{\La}^1(\coker b_3,\ker d_2)$, that is $\Theta(\alpha \cdot [f]) = \Theta([f])$.
\end{lemma}

\begin{proof}
	Let $\alpha\colon\coker b_3\oplus\La^n\to \coker b_3\oplus\La^n$ be a stable automorphism of $\coker b_3$. This can be lifted to a chain map
	\[\xymatrix{
		\cdots\ar[r]&B_3\ar[d]^{\alpha_3}\ar[r]^-{(b_3,0)^T}&	B_2\oplus\La^n \ar[d]^{\alpha_2}\ar[r]&\coker b_3\oplus\La^n\ar[d]^\alpha\\
		\cdots\ar[r]&B_3\ar[r]_-{(b_3,0)^T}&B_2\oplus\La^n\ar[r]&\coker b_3 \oplus \La^n
	}\]
	since the top row is projective and the bottom row is exact.
	The action of $\alpha$ on an extension represented by $f \colon B_3 \to \ker d_2$ is given by precomposition with $\alpha_3 \colon B_3 \to B_3$.
	We have the following chain map:
	\[\xymatrix @C+1cm{
		\cdots\ar[r]&B_3\ar[r]^-{(f\circ \alpha_3,b_3,0)^T}\ar[d]^{\alpha_3} & C_2\oplus (B_2\oplus \La^n)\ar[r]^-{(d_2,0,0)}\ar[d]^-{\left(\begin{smallmatrix}\id&0\\0&\alpha_2\end{smallmatrix}\right)} & C_1\ar[r]^-{d_1}\ar[d]^-{\id} & C_0 \ar[d]^-{\id} \\
		\cdots\ar[r]&B_3 \ar[r]^-{(f,b_3,0)^T} &  C_2\oplus (B_2\oplus\La^n)\ar[r]^-{(d_2,0,0)} &
		C_1\ar[r]^-{d_1} & C_0.
	}\]
	It remains to prove that the chain map above induces an isomorphism on second homology, which since it  is chain map between bounded chain complexes of f.g.\ projective module, implies that it is a chain homotopy equivalence.
	
	To see surjectivity on second homology, consider a pair $(x,y,\lambda)\in C_2\oplus B_2\oplus\La^n$ with $x\in \ker d_2$. Since $\alpha$ is an isomorphism, there exists $(y',\lambda')\in B_2\oplus \La^n$ and $ a\in B_3$ with $(y,\lambda)=(b_3(a),0)+\alpha_2(y',\lambda')$. In $H_2(C_f \oplus \La^n[2])$, i.e.\ the homology of the bottom row, we have
	\[[(x,y,\lambda)]=[x,b_3(a),0]+[0,\alpha_2(y',\lambda')]=[x-f(a),\alpha_2(y',\lambda')],\]
	which is the image of $(x-f(a),y',\lambda')$ under the above chain map.
	
	Now, to prove injectivity on second homology, consider a pair $(x,y,\lambda)\in C_2\oplus B_2\oplus\La^n$ with $x\in \ker d_2$, and assume that there exists $a\in B_3$ with $f(a)=x$ and $(b_3(a),0)=\alpha_2(y,\lambda)$. Again since $\alpha$ is an isomorphism, this implies that $\lambda=0$ and that there exists $a'\in B_3$ with $b_3(a')=y$. We have
	\[(b_3(a),0) =\alpha_2(y,0)=\alpha_2 (b_3(a'),0) =(b_3 \circ \alpha_3(a'),0) \in B_2 \oplus \La^n.\]
	Since $B_4\xrightarrow{\,b_4\,} B_3\xrightarrow{\,b_3\,} B_2$ is exact at $B_3$, there is an element $c\in B_4$ with $b_4(c) = a-\alpha_3(a')$. Since $f\circ b_4=0$, we have that $x=f(a)=f(b_4(c) + \alpha_3(a') = f(\alpha_3(a'))$. Hence $(x,y,0)=((f\circ\alpha_3)(a'),b_3(a'),0)$ and the element $(x,y,0)$  is trivial in second homology as desired.
\end{proof}

Let $\sAut(\coker b_3)$ denote the group of stable automorphisms of $\coker b_3$ as above.
We can now state the main theorem of this section.

\begin{thm} \label{thm:injective}
	The assignment $\Theta$ from~\eqref{eqn:Theta} descends to an injective map \[\Theta \colon \Ext_{\La}^1(\coker b_3,\ker d_2)/\sAut(\coker b_3) \ra \sCh_2(\pi).\]
\end{thm}

\begin{proof}
By \cref{lem:theta-well-def-cokernel}, $\Theta$ is well-defined on the quotient by $\sAut(\coker b_3)$.
	Let $f,g \colon B_3 \to \ker d_2$ represent two extensions, and suppose that their images in $\sCh_2(\pi)$ agree.
	Then there is a chain map
	\[\xymatrix @C+1cm{
		B_3\ar[r]^-{(f,b_3,0)^T}\ar[d]^-{h_3} &  C_2\oplus B_2\oplus \La^n\ar[r]^-{(d_2,0,0)}\ar[d]^{h_2}& C_1\ar[r]^-{d_1}\ar[d]^{h_1} & C_0\ar[d]^{h_0}\ar[r]&\Z\ar[d]^\id \\
		B_3\ar[r]_-{(g,b_3,0)^T} & C_2\oplus B_2\oplus\La^{n}\ar[r]_-{(d_2,0,0)} & C_1\ar[r]_-{d_1} & C_0\ar[r]&\Z
	}\]
	that induces an isomorphism on second homology. We can assume that the chain map $h_*$ is the identity on $C_1$ and $C_0$.
	
	Consider the diagram with exact rows:
	\[\xymatrix{
	0 \ar[r] & \ker (d_2,0,0)/\im (f,b_3,0)^T\ar[r]\ar[d]^{(h_2)_*}&\coker (f,b_3,0)^T\ar[r]^{(d_2,0,0)}\ar[d]^{(h_2)_*}&\im (d_2,0,0)\ar[d]^{\id} \ar[r] & 0 \\	
0 \ar[r] & \ker (d_2,0,0)/\im (g,b_3,0)^T\ar[r] & \coker (g,b_3,0)^T \ar[r]^{(d_2,0,0)} & \im (d_2,0,0) \ar[r] & 0.	
}\]
	Since $h_2$ induces an isomorphism on second homology, the map $\coker (f,b_3,0)^T\to \coker (g,b_3,0)^T$ induced by $h_2$ is an isomorphism by the five lemma.
	Use the facts that the composition $B_4\to B_3\xrightarrow{f}C_2$ is trivial and the sequence $B^2\to B^3\to B^4$ is exact, to see that the dual of $f$ lifts to a map $\widehat{f} \colon C^2\to B^2$, as in the next diagram.
	\[\xymatrix{ & C^2 \ar[d]^{f^*} \ar@{-->}[dl]_{\widehat{f}^*} & \\ B^2 \ar[r]^{b^3} & B^3 \ar[r] & B^4}\]
Dualise again to deduce that $f\colon B_3\to C_2$ factors as $B_3\xrightarrow{b_3}B_2\xrightarrow{\widehat f}C_2$. This gives rise to a commutative square
	\[\xymatrix @C+0.5cm{
		B_3 \ar[r]^\id\ar[d]_{(0,b_3)^T} & B_3\ar[d]^{(f,b_3)^T}\\
		C_2\oplus B_2\ar[r]^{\left(\begin{smallmatrix}
			\id &\widehat f\\
			0 &\id
			\end{smallmatrix}\right)} & C_2\oplus B_2
	}\]
	that induces an isomorphism $C_2\oplus\coker b_3\cong\coker (0,b_3)^T\to \coker(f,b_3)^T$.
	
	Add the identity on $\La^n$  to obtain an isomorphism
	\[
	C_2\oplus\coker b_3 \oplus \La^n \cong \coker(f,b_3,0)^T.
	\]
	Similarly, we obtain an isomorphism $C_2\oplus\coker b_3 \oplus \La^n \cong \coker(g,b_3,0)^T$.

In the next diagram we show a composition of these maps that give a stable automorphism of $\coker b_3$ and its resolution.
\[
\xymatrix @R+0.1cm @C+0.1cm{
\cdots \ar[r] & B_3 \ar[r]^-{(0,b_3,0)^T} \ar[d]_{\Id} & C_2 \oplus B_2 \oplus \La^n \ar[r] \ar[d]^{\bsm
  \Id & \widehat{f} & 0 \\ 0 & \Id & 0 \\ 0 & 0 & \Id \esm} & C_2 \oplus \coker b_3 \oplus \La^n \ar[r] \ar[d]^{\cong} & 0 \\
\cdots \ar[r] & B_3 \ar[r]^-{(f,b_3,0)^T} \ar[d]_{h_3} & C_2 \oplus B_2 \oplus \La^n \ar[r] \ar[d]^{h_2} &\coker (f,b_3,0)^T \ar[r] \ar[d]^{\cong} & 0 \\
\cdots \ar[r] & B_3 \ar[r]^-{(g,b_3,0)^T} \ar[d]_{\Id} & C_2 \oplus B_2 \oplus \La^n \ar[r] \ar[d]^{\bsm
  \Id & -\widehat{g} & 0 \\ 0 & \Id & 0 \\ 0 & 0 & \Id \esm} & \coker (g,b_3,0)^T \ar[r] \ar[d]^{\cong} & 0 \\
\cdots \ar[r] & B_3 \ar[r]^-{(0,b_3,0)^T} & C_2 \oplus B_2 \oplus \La^n \ar[r] & C_2 \oplus \coker b_3 \oplus \La^n \ar[r]  & 0
}
\]
The induced action on $B_3$ is via $h_3$, so the stable automorphism shown act on a map $f \colon B_3 \to \ker d_2$ representing an element of $\Ext^1_{\La}(\coker b_3,\ker d_2)$ by pre-composition with $h_3$, that is $f \circ h_3 \colon B_3 \to \ker d_2$.

	To recap, we started with the chain map $h_*$ arising from the assumption of two extensions having equal image in $\sCh_2(\pi)$, and we obtained an automorphism of $\coker b_3$, that acts on a representative map $B_3 \to \ker d_2$ in $\Ext^1_{\La}(\coker b_3,\ker d_2)$ by pre-composition with $h_3$.  To complete the proof, we have to show that $g\circ h_3$ and $f$ represent the same class in $\Ext^1_{\La}(\coker b_3,\ker d_2)$.
	
	Recall the definition of $\widehat f \colon B_2 \to C_2$ from above, and consider the composition
	\[F \colon B_2\xrightarrow{(\widehat f,\id,0)^T}C_2\oplus B_2 \oplus \La^n \xrightarrow{\,h_2\,} C_2\oplus B_2\oplus \La^n \xrightarrow{\,\pr_1\,} C_2.\]
	Now consider the commutative diagram
	\[\xymatrix{
		B_3\ar[d]^{b_3}\ar[r]^\id & B_3\ar[d]^{(f,b_3,0)^T}\ar[r]^{h_3}&B_3\ar[d]^{(g,b_3,0)^T}\ar[r]^\id&B_3\ar[d]^{g}\\
		B_2\ar[r]^-{(\widehat f,\id,0)^T} & C_2\oplus B_2\oplus\La^n\ar[d]^{(d_2,0,0)}\ar[r]^{h_2} & C_2\oplus B_2\oplus \La^{n}\ar[r]^-{\pr_1}\ar[d]^{(d_2,0,0)}& C_2\ar[d]^{d_2}\\
		&C_1\ar[r]^\id&C_1\ar[r]^\id&C_1}\]
	Commutativity of the three top squares shows that $g\circ h_3=F\circ b_3$.  Commutativity of the lower two squares pre-composed with $(\widehat f,\id,0)^T \colon B_2 \to C_2\oplus B_2\oplus\La^n$ proves that $d_2\circ F=d_2\circ \widehat f$. It follows from the first statement that $g\circ h_3-f=  F \circ b_3 -f$, which equals $(F-\widehat f)\circ b_3\colon B_3\to C_2$ since $f = \widehat{f} \circ b_3$.  This is in fact a map $(F-\widehat f)\circ b_3\colon B_3\to \ker d_2$ since $d_2\circ F=d_2\circ \widehat f$.  Thus $g\circ h_3-f \colon B_3 \to \ker d_2$ factors through $B_2$ and hence represents the trivial extension class.
\end{proof}

Next we check that the map $\Theta \colon \Ext_{\La}^1(\coker b_3,\ker d_2) \to \sCh_2(\pi)$ from~\eqref{eqn:Theta} does not depend on the choice of resolution $C_*$.
Let $(C'_*,d'_*)$ be another free resolution of $\Z$ as a $\La$-module with $C'_i$ finitely generated for $i=0,1,2$ and let $\alpha_*\colon C_*\to C'_*$ be a chain homotopy equivalence over $\Z$. Then $\alpha_*$ induces a map $a\colon \ker d_2\to \ker d_2'$ given by restricting $\alpha_2$.

\begin{lemma}\label{lem:theta-well-def-kernel}
	The diagram
	\[\xymatrix{
	\Ext_{\La}^1(\coker b_3,\ker d_2)\ar[r]^-\Theta\ar[d]^{a_*}&\sCh_2(\pi)\\
	\Ext_{\La}^1(\coker b_3,\ker d_2')\ar[ur]^\Theta&
}\]
commutes.
\end{lemma}

\begin{proof}
	First consider the case that $\alpha_*$ restricted to degrees $0,1,2$ is a chain homotopy equivalence. Then
		\[\xymatrix @C+1cm {
		\cdots\ar[r]&B_3\ar[d]^\id \ar[r]^-{(f,b_3)^T} &  C_2\oplus B_2\ar[d]^-{\alpha_2 \oplus \id} \ar[r]^-{(d_2,0)} & C_1 \ar[d]^-{\alpha_1}\ar[r]^-{d_1} & C_0\ar[d]^-{\alpha_0} \\
		\cdots\ar[r]&B_3\ar[r]^-{(\alpha_2\circ f, b_3)^T} & C_2'\oplus B_2 \ar[r]^-{(d_2',0)}  &   C_1' \ar[r]^-{d_1'} & C_0'.
	}\]
is a chain homotopy equivalence and $\Theta([f])=\Theta(a_*[f])$.

Next we will now reduce the lemma to the case $C_*=C'_*$. Let $(F_i\xrightarrow{\id} F_i)[i]$ denote the chain complex $F_i\xrightarrow{\id}F_i$ concentrated in degrees $i$ and $i+1$. Applying Schanuel's lemma \cite[\href{https://stacks.math.columbia.edu/tag/00O3}{Tag 00O3}]{stacks-project}, \cite[Lemma~VIII.4.2]{brown} inductively, for any two free resolutions $C_*,C_*'$ of $\Z$, there are free $\La$-modules $F_i$ and $F_i'$ such that
$C_*\oplus\bigoplus_{i\in \N_0}(F_i\xrightarrow{\id} F_i)[i]$ and $C'_*\oplus\bigoplus_{i\in \N_0}(F'_i\xrightarrow{\id} F'_i)[i]$ are isomorphic over $\Z$.  The inclusion of $C'_*\oplus (F'_2\xrightarrow{\id}F'_2)[2]$ into $C'_*\oplus\bigoplus_{i\in \N_0}(F'_i\xrightarrow{\id} F'_i)[i]$ as well as the projection of $C_*\oplus\bigoplus_{i\in \N_0}(F_i \xrightarrow{\id} F_i)[i]$ onto $C_*\oplus(F_2\xrightarrow{\id}F_2)[2]$ are chain homotopy equivalences that remain chain homotopy equivalences upon  restricting to degrees $0,1,2$. It follows from the first paragraph that there is a chain homotopy equivalence $\beta\colon C'_*\oplus (F'_2\xrightarrow{\id}F'_2)[2]\to C_*\oplus(F_2\xrightarrow{\id}F_2)[2]$ over $\Z$ such that, taking $b\colon \ker d_2'\to \ker d_2$ to be the restriction of $\beta_2$, we have $\Theta=\Theta\circ b_*$.

Next we want to relate this to the chain complexes $C_*$ and $C'_*$ before stabilisation.
The inclusion $C_* \to C_* \oplus (F_2\xrightarrow{\id}F_2)[2]$ is a chain homotopy equivalence whose induced map on kernels is the inclusion $\ker d_2 \to \ker d_2 \oplus F_2$, and similarly for $C'_*$.
By definition,
\[\xymatrix @R-0.4cm{\cdots\ar[r]& B_3\ar[r]^-{(f,0,b_3)^T} & C_2\oplus F_2\oplus B_2 \ar[r]^-{(d_2,0,0)}  &   C_1 \ar[r]^-{d_1} & C_0 & \text{and} \\
\cdots\ar[r] & B_3\ar[r]^-{(f,b_3)^T} & C_2\oplus B_2 \ar[r]^-{(d_2,0)}  &   C_1 \ar[r]^-{d_1} & C_0 & }\]
represent the same element in $\sCh_2(\pi)$.
Thus composing $C'_* \to C'_* \oplus (F'_2\xrightarrow{\id}F'_2)[2] \xrightarrow{\beta} C_* \oplus (F_2 \xrightarrow{\id} F_2) \to C_*$ gives  a chain homotopy equivalence $\beta^\dag\colon C'_*\to C_*$ over $\Z$ with induced map $b^\dag_*$ such that $\Theta=\Theta\circ b^\dag_*$.

Given any chain homotopy equivalence $\alpha\colon C_*\to C'_*$ over $\Z$ we can consider the composition $\beta^\dag \circ\alpha$. Since $\Theta=\Theta\circ b^\dag_*$ for $\beta^\dag$, it suffices to show that $\Theta=\Theta\circ b^\dag_* \circ a$.  So we may restrict to the case $C_*'=C_*$.

Let $\alpha_*\colon C_*\to C_*$ be a chain homotopy equivalence over $\Z$ and let $a\colon \ker d_2\to \ker d_2$ be the restriction of $\alpha_2$. As $\alpha_*$ is a chain homotopy equivalence over $\Z$, it is chain homotopic to the identity. Let $H_*\colon C_*\to C_{*+1}$ be such that $\id-\alpha_*=d_{*+1}H_*+H_{*-1}d_*$. Then the diagram on the left commutes. Thus also the diagram on the right commutes.
\[\xymatrix{
	\ker d_2\ar[r]^{a}\ar[d]&\ker d_2\ar[d]  &&  	\ker d_2\ar[r]^{a}\ar[d]&\ker d_2\ar[d] \\
	C_2\ar[r]^{\alpha_2+H_1d_2}\ar[d]^{d_2}&C_2\ar[d]^{d_2} && 	 C_2\ar[r]^{\alpha_2+H_1d_2}\ar[d]^{d_2}&C_2\ar[d]^{d_2}\\
	C_1\ar[r]^{\id}&C_1 && 	\im d_2\ar[r]^{\id}&\im d_2
}\]
On the right the columns are short exact. Apply $\Ext^*_\La(\coker b_3,-)$ to obtain the following diagram.
\[\xymatrix{
	\cdots\ar[r]&\Hom_\La(\coker b_3,\im d_2)\ar[r]\ar[d]^\id&\Ext^1_\La(\coker b_3,\ker d_2)\ar[d]^{a_*}\ar[r]&0\\
	\cdots\ar[r]&\Hom_\La(\coker b_3,\im d_2)\ar[r]&\Ext^1_\La(\coker b_3,\ker d_2)\ar[r]&0
}\]
Here we use that $\Ext^1_\La(\coker b_3,C_2)=0$, which since $C_2$ is free follows from the assumption on $b_3$ that  $\Ext^1_\La(\coker b_3,\La)=0$.
In particular, $a_*\colon \Ext^1_\La(\coker b_3,\ker d_2)\to \Ext^1_\La(\coker b_3,\ker d_2)$ is the identity.
\end{proof}

Now, to apply this to the $\CP$-stable classification of 4-manifolds, we need to investigate the action of $\sAut(\coker b_3)$ on the extension classes for some families of 1-types. More precisely, we choose $b_3$ to be $d^2_w$, the dual of $d_2$ twisted by $w\colon\pi\to \Z/2$.

\begin{lemma}
	\label{lem:autcoker}
	Assume that $H^1(\pi;\La)=0$ and $\pi$ is infinite. Then $\sAut(\coker d^2_w)$ acts on $\Ext_{\La}^1(\coker d^2_w,\ker d_2)$ by multiplication by $\pm 1$.
\end{lemma}

\begin{proof}
	Since $\pi$ is infinite, $H^0(\pi;\La)=0$. As $\La$ is isomorphic to $\La^w$ as a right $\La$-module, we also have $H^0(\pi;\La^w)=H^1(\pi;\La^w)=0$. Hence the sequence
	\[0\ra C^0\xrightarrow{d_w^1}C^1\xrightarrow{(d_w^2,0)^T}C^2\oplus \La^n\ra \coker d_w^2\oplus \La^n\ra 0\]
	is exact. Thus every stable automorphism $\alpha$ of $\coker d_w^2$ lifts to a chain map $\alpha^*$ as follows.
	\[\xymatrix @C+0.3cm{
		C^0\ar[d]_{\alpha^0} \ar[r]^-{d_w^1} & C^1\ar[d]_-{\alpha^1} \ar[r]^-{(d_w^2,0)^T} & C^2\oplus\La^n \ar[d]_-{\alpha^2} \ar[r] & \coker d_w^2\oplus\La^n \ar[d]_-{\alpha}\\
		C^0\ar[r]^-{d_w^1} &  C^1 \ar[r]^-{(d_w^2,0)^T}  & C^2\oplus \La^n \ar[r] & \coker d_w^2 \oplus \La^n.
	}\]
	The action of $\alpha$ on $\Ext^1_{\La}(\coker d_w^2,\ker d_2)$ is given by pre-composition with $\alpha^1$. Let $\beta^*$ be a lift of $\alpha^{-1}$ to a chain map. Let $H^{*}\colon C^*\to C^{*-1}$ be a chain homotopy from $\beta^*\circ\alpha^*$ to $\id$. In particular, $\beta^0\circ\alpha^0-\id=H^1\circ d_w^1$. Take duals and twist with $w$ to obtain $\alpha_0\circ\beta_0-\id=d_1\circ H_1$. Hence $\alpha_0\circ \beta_0$ induces the identity on $\coker d_1\cong \Z$ and thus $\alpha_0$ induces multiplication by plus or minus one on $\coker d_1$.
	
	Dualise $\alpha^*$ and twist with $w$ again to obtain the following diagram, where the maps $\alpha_i$ are the maps dual to $\alpha^i$.
	\[\xymatrix{
		C_0&C_0\ar[l]_-{\alpha_0} \ar[dl]_-{G_0}\\
		C_1\ar[u]^{d_1}&C_1\ar[u]_{d_1}\ar[l]_{\alpha_1} \ar[dl]_-{G_1}\\
		C_2\oplus\La^n\ar[u]^{(d_2,0)}&C_2\oplus\La^n\ar[u]_{(d_2,0)}\ar[l]_-{\alpha_2}\\
	}\]
	Since $\alpha_0$ induces multiplication by plus or minus one on $\coker d_1$, $\alpha_0\pm \id\colon C_0\to C_0$ factors through $d_1$. That is, there exists a homomorphism $G_0\colon C_0\to C_1$ with $d_1 \circ G_0=\alpha_0\pm\id$. This implies $d_1\circ((\alpha_1\pm\id)-G_0\circ d_1)=0$, so there exists $G_1\colon C_1\to C_2 \oplus \La^n$ with $\alpha_1\pm \id=G_0\circ d_1+(d_2,0)\circ G_1$.
	
	Take duals and twist by $w$ one last time to obtain $\alpha^1\pm \id=d_w^1\circ G^0+G^1\circ d_w^2$, where $G^i$ is the map dual to $G_i$ for $i=0,1$.
	
Now let $f\colon C^1\to \ker d_2$ represent an extension $\Ext_{\La}^1(\coker d_w^2,\ker d_2)$. In particular, $f\circ d_w^1=0$ and thus
\[\alpha_*[f]=[f\circ \alpha^1]=[\mp f+f\circ G^1\circ d_w^2]=[\mp f].\]
Hence $\alpha$ acts by multiplication by $\pm 1$ on $\Ext_{\La}^1(\coker d_w^2,\ker d_2)$. This completes the proof of \cref{lem:autcoker}.
\end{proof}

\section{Proof of \texorpdfstring{\cref{thm:k}}{Theorem~A}}
\label{sec:k-ext}
Let $\hCW_*$ denote the category whose objects are based connected CW-complexes and whose morphisms are based homotopy classes of maps.
Let $\hCW_2$ be the full subcategory of based, connected, and $3$-coconnected (i.e.\ homotopy groups $\pi_i$ vanish for $i \geq 3$) CW-complexes.
Taking the second stage of the Postnikov tower gives a functor $P_2 \colon \hCW_*\to \hCW_2$.  As discussed in the introduction, the $k$-invariant $k(X)\in H^3(\pi_1(X);\pi_2(X))$ of $X\in CW_*$ classifies the fibration
\[
K(\pi_2(X),2) \ra P_2(X) \ra K(\pi_1(X),1).
\]
Let $\hCW_2(\pi)$ be the full subcategory of $\hCW_2$ of objects with fundamental group $\pi$. For each $X\in \hCW_2(\pi)$, define $\gimel(X)\in\sCh_2(\pi)$ as follows: take a $\La$-module chain complex $(C_*,d_*)$ for $X$ and replace it in degrees $\geq 4$ by a free resolution of $\ker d_3$.
\begin{prop}
	\label{prop:k}
	Let $M$ be a closed $4$-manifold with 1-type $(\pi,w)$ and chain complex $(C_* = C_*(M;\La),d_*)$. The composition
	\[ H_4(\pi;\Z^w)\ra \Ext_{\La}^1(\coker d_3,\ker d_2)/\sAut(\coker d_3)\xrightarrow{\,\Theta\,}\sCh_2(\pi)\]
	sends the $($twisted$)$ fundamental class $c_*[M]$ to $\gimel(P_2(M))$.
\end{prop}

\begin{proof}
	Fix a $4$-manifold $M$ with $\pi_1(M) =\pi$ and choose a handle decomposition of~$M$. Denote the associated $\La$-chain complex by $(C_*,d_*)$.  We apply the theory from Section~\ref{sec:k-ext1} with $\cdots \to B_4 \to B_3 \xrightarrow{b_3} B_2$ as $\cdots\to B_4 \to C_3 \xrightarrow{d_3} C_2$, a free resolution for $\coker(d_3)$ beginning with $d_3 \colon C_3 \to C_2$.
 Any two choices of handle decomposition induce chain equivalent chain complexes $C_*$, and stably equivalent modules $\coker d_3$. It follows from \cref{lem:theta-well-def-cokernel} and \cref{lem:theta-well-def-kernel} that the image of $c_*[M]$ under the map in the statement of the proposition does not depend on the choice of handle decomposition of~$M$.

	Consider $\Theta$ applied to the extension class of $\pi_2(M)$ in $\Ext^1_{\La}(\coker d_3,\ker d_2)$. Since this extension class is represented by the homomorphism $d_3\colon C_3 \to C_2$, we consider the chain complex $C_{d_3}$:
	\[\xymatrix{\cdots\to C_3 \ar[r]^-{(d_3,d_3)^T}&C_2\oplus C_2\ar[r]^-{(d_2,0)}&C_1\ar[r]^{d_1}&C_0.}\]

	Use the isomorphism $\left(\begin{smallmatrix}\id&0\\-\id&\id\end{smallmatrix}\right) \colon C_2\oplus C_2 \to C_2 \oplus C_2$ to see that the previously displayed chain complex is chain isomorphic to the chain complex $\gimel(P_2(M))\oplus \La^{\rank(C_2)}[2]$ (replace the second $d_3$ in $(d_3,d_3)$ with $0$). Hence in $\sCh_2(\pi)$, $\Theta([d_3])$ agrees with $\gimel(P_2(M))$ as desired.
\end{proof}

Note that each stable automorphism of $\coker b_3$ induces an automorphism of the extension group and hence the trivial element is fixed by the action. Therefore, combining \cref{thm:injective} and \cref{prop:k} with \cref{thm:prim-obstruction} yields the following corollary.
\begin{cor}
	\label{cor:detect-trivial}
	The $\CP$-stable diffeomorphism class represented by the trivial extension is detected by the stable 2-type $(\pi,w,\pi_2,k)$.
	More precisely, let $K$ be a 2-complex representing $(\pi,w)$.  For every 1-type $(\pi,w)$, a closed 4-manifold $M$ with this 1-type is $\CP$-stably diffeomorphic to the double $N:= \nu K \cup - \nu K$ if and only if the 2-types $[\pi_2(M) ,k(M)] = [\pi_2(N),k(N)]$ are stably isomorphic.
\end{cor}

\begin{proof}[Proof of \cref{thm:k}]
	By \cite[Proposition 13.5.3 and Theorem 13.5.5]{geoghegan-2008}, a group~$\pi$ with one end is infinite and has $H^1(\pi;\Lambda)=0$. Hence \cref{thm:k}~\eqref{thm:k:item2} follows by combining \cref{thm:prim-obstruction} with \cref{thm:injective}, \cref{prop:k} and \cref{lem:autcoker}.
	In more detail, recall that \cref{thm:prim-obstruction} says that the extension class of $\pi_2(M) $ in  $\Ext(H^2(K;\La^w),H_2(K;\La)) \cong H_4(\pi;\Z^w)$ determines $c_*[M]$ up to sign, which by \cref{theorem:Kreck} determines $\CP$-stable diffeomorphism.  By \cref{thm:injective} and \cref{prop:k} the pair $(\pi_2(M) ,k(M))$ determines the extension class, up to the action of $\sAut(\coker d_w^2)$.  By \cref{lem:autcoker}, the action of $\sAut(\coker d_w^2)$ is just multiplication by $\pm 1$, so altogether $(\pi_2(M) ,k(M))$ determines $c_*[M]$ up to sign.
	
If $\pi$ is torsion-free, then by Stalling's theorem \cite[Theorems~4.11~and~5.1]{stallings}, $\pi$ has more than one end if and only if $\pi\cong \Z$ or $\pi$ is a free product of two non-trivial groups. For $\pi=\Z$, $H_4(\pi;\Z^w)=0$, hence every pair of closed $4$-manifolds is with a fixed 1-type is $\CP$-stably diffeomorphic.  Hence the conclusion of \cref{thm:k} holds for $\pi=\Z$ for trivial reasons.
	
	For $\pi\cong G_1\ast G_2$, we have $H_4(\pi;\Z^w)\cong H_4(G_1;\Z^{w_1})\oplus H_4(G_2;\Z^{w_2})$, where $w_i$ denotes the restriction of $w$ to $G_i$, $i=1,2$. Hence $\CP$-stably any closed $4$-manifold with fundamental group $\pi$ is the connected sum of manifolds with fundamental group $G_1$ and $G_2$. Therefore, \cref{thm:k}~\eqref{thm:k:item1} follows in this case by induction, with base cases $\pi=\Z$ or the groups with one end of \eqref{thm:k:item2}.
	
	If $\pi$ is finite then $H^1(\pi;\Lambda)=0$ by \cite[Proposition 13.3.1]{geoghegan-2008}. If multiplication by 4 or 6 annihilates $H_4(\pi;\Z^w)$, then the subgroup generated by $c_*[M]$ is cyclic of order 2, 3, 4 or 6. In each case it has a unique generator up to sign and hence it determines $c_*[M]$ up to a sign.  As a consequence, \cref{thm:k}~\eqref{thm:k:item3} follows from \cref{cor:hopfseq}, which shows that the subgroup generated by $c_*[M]$ is determined by the 2-type of $M$.
\end{proof}

\section{Fundamental groups of aspherical \texorpdfstring{$4$}{4}-manifolds} \label{sec:aspherical}
In this section, we fix a closed, connected, aspherical $4$-manifold  $X$ with orientation character $w$ and fundamental group $\pi$. We can identify
\[
H_4(\pi;\Z^w)/\pm\Aut(\pi) \cong H_4(X;\Z^w)/\pm\Aut(\pi) \cong \Z/\pm \cong \N_0
\]
and write $|c_*[M]|$ for the image of $c_*[M]$ under this sequence of maps.

\begin{thm} \label{thm:aspherical}
 Let $M$ be a closed $4$-manifold with 1-type $(\pi,w)$ and classifying map $c\colon M \to X=B\pi$.
\begin{enumerate}
\item If $|c_*[M]|\neq 0$, then $H_1(\pi;\pi_2(M) ^w)$ is a cyclic group of order $|c_*[M]|$.
\item If $|c_*[M]|=0$, then $H_1(\pi;\pi_2(M) ^w)$ is infinite.
\end{enumerate}
Thus two closed 4-manifolds $M_1$ and $M_2$ with fundamental group $\pi$ and orientation character $w$ are $\CP$-stably diffeomorphic if and only if they have stably isomorphic second homotopy groups $\pi_2(M_1) \oplus \La^{r_1} \cong \pi_2(M_2) \oplus \La^{r_2}$ for some $r_1,r_2 \in \mathbb{N}_0$.
\end{thm}

\begin{proof}
For the proof, we fix twisted orientations on $M$ and $X$. In particular this determines an identification $H_4(X;\Z^w) =\Z$.

Since $c\colon M \to X$ induces an isomorphism on fundamental groups, the map $c^*\colon H^1(X;\Z)\to H^1(M;\Z)$ is an isomorphism. Consider the following commutative square.
\[\xymatrix{
H_3(M;\Z^w)\ar[r]^-{c_*}&H_3(X;\Z^w)\\
H^1(M;\Z)\ar[u]^{-\cap[M]}_\cong&H^1(X;\Z)\ar[l]_-{c^*}^-\cong\ar[u]_{-\cap c_*[M]}
}\]
Since $H^1(X;\Z)$ is torsion free and capping with $[X]$ is an isomorphism, capping with $c_*[M]$ is injective if $c_*[M]\neq 0$. Thus from \cref{thm:hopfseq} we directly obtain the isomorphism $\Z/|c_*[M]| \cong H_1(\pi;\pi_2(M) ^w)$.

If $c_*[M]=0$, then it also follows from \cref{thm:hopfseq} that $\Z\cong H_4(X;\Z^w)$ is a subgroup of $H_1(\pi;\pi_2(M) ^w)$, and hence the latter is infinite.
\end{proof}

We can also show that under certain assumptions, $|c_*[M]|$ can be extracted from the stable isomorphism class of $\Z^w\otimes_{\La}\pi_2(M) $, an abelian group that is frequently much easier to compute than the $\La$-module $\pi_2(M) $.

\begin{thm}
	\label{thm:priexample}
In the notation of Theorem~\ref{thm:aspherical}, assume that $H^1(X;\Z)\neq 0$. Then
\begin{enumerate}
\item If $c_*[M]\neq 0$, then $|c_*[M]|$ is the highest torsion in $\Z^w\otimes_{\La}\pi_2(M) $, that is the maximal order of torsion elements in this abelian group.
\item If $H^2(X;\Z)$ has torsion, then  $|c_*[M]|$ is completely determined by the torsion subgroup of
$\Z^w\otimes_{\La}\pi_2(M) $. This torsion subgroup is trivial if and only if $|c_*[M]|=1$.
\item If $H^2(X;\Z)$ is torsion-free, then  $|c_*[M]|=1$ if and only if $\pi_2(M) $ is projective, and otherwise $|c_*[M]|$ is completely determined by
the torsion subgroup of $\Z^w\otimes_{\La}\pi_2(M) $. This torsion subgroup is trivial if and only if $|c_*[M]|\in \{0,1\}$.
\end{enumerate}
\end{thm}

Except for the statement that $c_*[M]$ corresponds to the highest torsion in $\Z^w\otimes_{\La}\pi_2(M) $, the above theorem can be proven more easily using the exact sequence from \cref{thm:hopfseq}. But since we believe that this statement is worth knowing, we take a different approach and start with some lemmas. For the lemmas we do not yet need the assumption that $H^1(\pi;\Z) = H^1(X;\Z) \neq 0$; we will point out in the proof of \cref{thm:priexample} where this hypothesis appears. Choose a handle decomposition of $X$ with a single $4$-handle and a single $0$-handle and let $(C_*,d_*)$ denote the $\La$-module chain complex of $X$ associated to this handle decomposition.

\begin{lemma}
	\label{lem:generator}
The group $\Ext_{\La}^1(\coker d_3,\ker d_2) \cong H_{4}(\pi;\Z^w)\cong \Z$ is generated by
 \[
 0\ra\ker d_2\xrightarrow{\,i\,} C_2\xrightarrow{\,p\,} \coker d_{3}\ra 0.
 \]
\end{lemma}

\begin{proof}
By \cref{prop:mainpri2} the extension $0 \to \ker d_2 \xrightarrow{i} C_2\xrightarrow{p} \coker d_{3}\to 0$ corresponds to $c_*[X] \in H_4(\pi;\Z^w)$, where $c\colon X\to B\pi$ is the map classifying the fundamental group. Since in this case $X$ is a model for $B\pi$, we can take $c = \id_X$ with $[X]$ a generator of $H_4(X;\Z^w)$.
\end{proof}

\begin{lemma}
	\label{lem:extension}
Using the generator from \cref{lem:generator}, the extension corresponding to $m\in \Z$ is given by
\[0\ra \ker d_2\xrightarrow{(0,\id)^T}(C_2\oplus\ker d_2)/\{(i(a),ma)\mid a\in \ker d_2\}\xrightarrow{p\circ p_1}\coker d_{3}\ra 0,\]
where $p_1$ is the projection onto the first summand.
\end{lemma}

Note that $m=0$ gives the direct sum $E_0 = \ker d_2 \oplus \coker d_3$.

\begin{proof}
	In the case $m=1$, the group $(C_2\oplus\ker d_2)/\{(i(a),a)\mid a\in \ker d_2\}$ is isomorphic to $(C_2\oplus \ker d_2)/\{(0,a)\mid a\in \ker d_2\}\cong C_2$, where this isomorphism is induced by the isomorphism $C_2\oplus \ker d_2\to C_2\oplus\ker d_2$ given by $(c,a)\mapsto (c-i(a),a)$. Under this isomorphism the extension from the lemma is mapped to the extension from \cref{lem:generator}. Hence it suffices to show that the Baer sum of two extension for $m,m'$ as in the lemma is isomorphic to the given extension for $m+m'$.
	
	Let $L$ be the submodule of $C_2\oplus \ker d_2\oplus C_2\oplus \ker d_2$ consisting of all $(c_1,a_1,c_2,a_2)$ with $p(c_1)=p(c_2)$, let $L'$ be the submodule
	\[L':=\{(i(a),ma+b,i(a'),m'a'-b)\mid a,a',b\in\ker d_2\}\]
	and let $E:=L/L'$.  The Baer sum~\cite[Definition~3.4.4.]{Weibel94}
	of the extensions for $m$ and $m'$ is given by
	\[0\ra \ker d_2\ra E\ra \coker d_3\ra 0,\]
	where the map $\ker d_2\to E$ is given by $a\mapsto [0,a,0,0]$ and the map $E\to \coker d_3$ is given by $[(c_1,a_1,c_2,a_2)]\mapsto p(c_1)$.
	The map
	\[\begin{array}{rcl}
	f \colon L & \ra & C_2\oplus \ker d_2\oplus\ker d_2\oplus \ker d_2 \\
	(c_1,a_1,c_2,a_2) & \mapsto & (c_1,a_1+a_2+m'(c_1-c_2),c_2-c_1,a_2)
	\end{array}\]
	defines an isomorphism, and the subset $L'$ is mapped to \begin{align*}f(L')&=\{(i(a),(m+m')a,a'-a,m'a'-b)\mid a,a',b\in \ker d_2\}\\&=\{(i(a),(m+m')a,a',b)\mid a,a',b\in\ker d_2\}.\end{align*}
	Hence $f$ induces an isomorphism
	\[E\cong (C_2\oplus (\ker d_2)^3)/f(L')\cong (C_2\oplus\ker d_2)/\{(i(a),(m+m')a)\mid a\in \ker d_2\}.\]
	This defines an isomorphism from the Baer sum to the extension for $m+m'$ from the statement of the lemma.
\end{proof}

\begin{lemma}
	\label{lem:Zextension}
	Let $0\to \ker d_2\to E_m\to \coker d_3\to 0$
	be the extension for $m\in \Z$ from \cref{lem:extension}. Then
	\[\Z^w\otimes_{\La}E_m\cong \frac{\big((\Z^w\otimes_{\La} C_2)\oplus (\Z^w\otimes_{\La} C_3)\big)}{\{(\id^w_\Z\otimes d_3)a,ma)\mid a\in \Z^w\otimes_{\La}C_3\}}.\]
\end{lemma}

\begin{proof}
	Note that $\id_{\Z^w}\otimes d_4=0$, since $X$ has orientation character $w$ and thus $H_4(X;\Z^w)\cong \Z^w\otimes_{\La}C_4\cong \Z$.
	By exactness, $\ker d_2= \im d_3\cong \coker d_4$. By right exactness of the tensor product, it follows that
	$\Z^w\otimes_{\La}\ker d_2\cong \coker(\id_{\Z^w}\otimes d_4)\cong \Z^w\otimes_{\La}C_3$. Tensor the diagram
	\[\xymatrix{
	C_3\ar@{->>}[rd]\ar[rr]^{d_3}&&C_2\\
	&\ker d_2\ar[ru]^{i}&	
	}\]
	with $\Z^w$ over $\La$, to obtain
	\[\xymatrix{
		\Z^w\otimes_{\La} C_3\ar[rd]^\cong\ar[rr]^{\id_{\Z^w} \otimes d_3}&&\Z^w\otimes_{\La}C_2\\
		&\Z^w\otimes_{\La}\ker d_2\ar[ru]^{\id_{\Z^w}\otimes i}&	
	}\]
	The lemma now follows from right exactness of the tensor product.
	\end{proof}

	\begin{proof}[Proof of \cref{thm:priexample}]
		Let $E_m$ be as in \cref{lem:Zextension}. By \cref{prop:mainpri1} and \cref{prop:mainpri2}, $\pi_2(M) $ is stably isomorphic to $E_m$ as a $\La$-module if $c_*[M]=\pm m$. All of the  conditions appearing in \cref{thm:priexample} are invariant under adding a free $\La$-summand to $\pi_2(M) $. Since for any extension and its negative, the middle groups are isomorphic, we have $E_m\cong E_{-m}$. Hence we can only obtain a distinction up to sign.
		
The boundary map $\id_{\Z^w} \otimes d_4 \colon \Z^w \otimes_{\La}C_4 \to \Z^w \otimes_{\La} C_3$ vanishes since $H_4(X;\Z^w) \cong \Z$ and $X$ has a unique $4$-handle.  Hence $H_3(X;\Z^w) \cong \ker(\Id_{\Z^w} \otimes d_3)$.  By Poincar\'{e} duality \[H^1(\pi;\Z) \cong H^1(X;\Z) \cong H_3(X;\Z^w) \cong \ker(\Id_{\Z^w} \otimes d_3).\]

	\begin{claim}
	The cokernel of $\Id_{\Z^w} \otimes d_3 \colon \Z^w \otimes_{\La} C_3 \to \Z^w \otimes_{\La} C_2$  is stably isomorphic to $H_2(\pi;\Z^w)\cong H^2(\pi;\Z)$.
	\end{claim}
		
		We have a cochain complex:
		\[\xymatrix @C+0.4cm {\Z^w \otimes_{\La} C_3 \ar[r]^-{\id_{\Z^w} \otimes d_3} & \Z^w \otimes_{\La} C_2 \ar[r]^-{\id_{\Z^w} \otimes d_2} & \Z^w \otimes_{\La} C_1}\]
		whose cohomology $\ker(\id_{\Z^w} \otimes d_2)/\im(\id_{\Z^w} \otimes d_3)$ is isomorphic to $H_2(\pi;\Z^w)$.  There is an exact sequence:
		\[0 \ra \frac{\ker(\id_{\Z^w} \otimes d_2)}{\im(\id_{\Z^w} \otimes d_3)} \ra \frac{\Z^w \otimes_{\La} C_2}{\im(\id_{\Z^w} \otimes d_3)} \ra \frac{\Z^w \otimes_{\La} C_2}{\ker(\id_{\Z^w} \otimes d_2)} \ra 0.\]
		Since $\Z^w \otimes_{\La} C_2/\ker(\id_{\Z^w} \otimes_{\La} d_2) \cong \im(\id_{\Z^w} \otimes d_3)$, and $\im(\id_{\Z^w} \otimes d_2)$ as a submodule of the free abelian group $\Z^w \otimes_{\La} C_1$ is free, we see that there is a stable isomorphism of the central group \[(\Z^w \otimes_{\La} C_2)/\im(\id_{\Z^w} \otimes d_3) = \coker(\Id_{\Z^w} \otimes d_3 \colon \Z^w \otimes_{\La} C_3 \ra \Z^w \otimes_{\La} C_2)\] with $H_2(\pi;\Z^w)$, as claimed.

		By applying elementary row and column operations, $\id_{\Z^w}\otimes d_3\colon \Z^w\otimes_{\La}C_3\to \Z^w\otimes_{\La}C_2$ can be written as
		\begin{equation}\label{eq:extension}\Z^a\oplus\Z^k\xrightarrow{\left(\begin{smallmatrix}
			0&0\\
			0&D
			\end{smallmatrix}\right)}\Z^b\oplus \Z^k,
		\end{equation}
		with $a\geq 1, k,b\geq 0$ and $D$ a diagonal matrix with entries $\delta_1,\ldots, \delta_k\in\bbN\setminus\{0\}$.
		In order to see that $a \geq 1$, we use that the kernel is $H^1(\pi;\Z)$ together with the hypothesis that $H^1(X;\Z) \cong H^1(\pi;\Z) \neq 0$.
		
		 Using the description from \cref{lem:Zextension}, stably
		 \[\Z^w\otimes_{\La} E_m\cong_s (\Z/m\Z)^a\oplus\bigoplus_{i=1}^k\Z/\mathrm{gcd}(\delta_i,m)\Z.\]
		 If $m\neq 0$, then the highest torsion in $\Z^w\otimes_{\La}E_m$ is $m$-torsion.  Thus $E_{m}$ is not isomorphic to $E_{m'}$ whenever $m \neq m'$ and both are nonzero.  It remains to distinguish $E_0$ from $E_m$ for $m\neq 0$ (recall $\gcd(\delta_i,0) = \delta_i$).
		
		 First let us assume that the torsion subgroup $TH^2(\pi;\Z)$ is nontrivial. Then the cokernel of \eqref{eq:extension} is not free, and hence there exists $1\leq j\leq k$ with $\delta_j>1$. Thus $\Z^w\otimes_{\La}E_0\cong \bigoplus_{i=1}^k\Z/\delta_i\Z$ is not torsion free. Since $\Z^w\otimes_{\La}E_m$ can only contain $\delta_i$-torsion if $\delta_i$ divides $m$, $\Z^w\otimes_{\La}E_0$ and $\Z^w\otimes_{\La}E_m$ can only be stably isomorphic if all the $\delta_i$ divide $m$. This already implies $|m|>1$. But in this case
		 \[\Z^w\otimes_{\La}E_m\cong (\Z/m\Z)^a\oplus \bigoplus_{i=1}^k\Z/\delta_i\Z\cong (\Z/m\Z)^a\oplus \Z^w\otimes_{\La}E_0,\]
		 and so $\Z^w\otimes_{\La}E_0$ and $\Z^w\otimes_{\La}E_m$ are not stably isomorphic, since $a \geq 1$. This completes the proof of case where $TH^2(\pi;\Z)$ is nontrivial.
		
		 If $H^2(\pi;\Z)$ is torsion free, then all the $\delta_i$ are $1$ and $\Z^w\otimes_{\La}E_m\cong_s (\Z/m\Z)^a$. This already distinguishes the cases $|m|>1$ and $m\in\{0,\pm 1\}$. For $m=\pm 1$, we have $E_m\cong C_2$ is free and for $m=0$ we have $E_0=\ker d_2 \oplus \coker d_3$. Note that $\ker d_2$ is not projective, since otherwise
		 \[0 \ra \ker d_2 \ra C_2 \ra C_1 \ra C_0 \] would be a projective $\La$-module resolution for $\Z$, which cannot be chain equivalent to the $\La$-module chain complex of a closed aspherical $4$-manifold with fundamental group $\pi$.  Therefore, $E_0$ is not projective.
	\end{proof}
	
\begin{rem}
	Note that the assumption that $H^1(X;\Z) \cong H^1(\pi;\Z)\neq 0$ is equivalent to the abelianisation $\pi_{\mathrm{ab}}$ being infinite, since $H^1(\pi;\Z)=\Hom_\Z(\pi_{\mathrm{ab}},\Z)$ and $\pi_{\mathrm{ab}}$ is finitely generated. This assumption is crucial; without it there exist $m>m'\geq 0$ with $\Z^w\otimes_{\La}E_m\cong \Z^w\otimes_{\La}E_{m'}$. For example this would happen for $m=\prod_{i=1}^k \delta_i$ and $m'=2m$, where the $\delta_i$ are as in the proof of \cref{thm:priexample}.
\end{rem}

\section{Examples demonstrating necessity of hypotheses}\label{section:examples-necessity}

In the preceding sections we saw that for large classes of finitely presented groups, such as infinite groups $\pi$ with $H^1(\pi;\La)=0$,  the quadruple $(\pi_1,w,\pi_2,k)$ detects the  $\CP$-stable diffeomorphism type.  Moreover for fundamental groups of aspherical $4$-manifolds even $(\pi_1,w,\pi_2)$ suffices. In this section we give examples where the data $(\pi_1,w,\pi_2,k)$ does not suffice to detect the $\CP$-stable diffeomorphism classification, showing that the hypothesis $H^1(\pi;\La)=0$ is required.   We also provide examples where the data is sufficient to detect the classification, but the $k$-invariant is relevant, so all of the data is necessary.

\subsection{The 2-type does not suffice in general}\label{section:post-2-type-does-not-suffice-lens-spaces}

In this section, as promised in \cref{sec:necessity-intro}, we give examples of orientable manifolds that are not $\CP$-stably diffeomorphic, but with isomorphic 2-types.

It will be helpful to recall the construction of 3-dimensional lens spaces $L_{p,q}$.  Start with the unit sphere in $\mathbb{C}^2$ and let $\xi$ be a $p$th root of unity. On the unit sphere, $\Z/p$ acts freely by $(z_1,z_2)\mapsto (\xi z_1,\xi^{q} z_2)$ for $0<q<p$ such that $p,q$ are coprime. The quotient of $S^3 \subset \mathbb{C}^2$ by this action is $L_{p,q}$.

Now fix an integer $p \geq 2$, let $\pi := \Z/p \times \Z$ and consider the 4-manifolds $N_{p,q} := L_{p,q} \times S^1$.    Note that $\pi_2(N_{p,q}) = \pi_2(\wt{N_{p,q}}) = \pi_2(S^3 \times \R) =0$.  Thus the stable Postnikov 2-type is trivial for all $p,q$, and the number $q$ cannot possibly be read off from the 2-type.

Nevertheless, we will show below that for most choices of $p$, there are $q,q'$ for which the resulting manifolds $N_{p,q}$ and $N_{p,q'}$ are not $\CP$-stably diffeomorphic. The smallest pair is $N_{5,1}$ and $N_{5,2}$, corresponding to the simplest homotopically inequivalent lens spaces $L_{5,1}$ and $L_{5,2}$ with isomorphic fundamental groups.  In general, compute using the equivalence of \eqref{lensit:1} and \eqref{lensit:4} of \cref{prop:TFAE} below.

Although $\pi$ is infinite, by the following lemma the group $\pi$ has two ends, so these examples are consistent with our earlier investigations.

\begin{lemma}
We have that $H^1(\pi;\La) \cong \Z$.
\end{lemma}

\begin{proof}
We compute using the fact that the manifold can be made into a model for $B\pi$ by adding cells of dimension $3$ and higher, which do not alter the first cohomology.
\[
H^1(\pi;\La) \cong H^1(N_{p,q};\La) \stackrel{PD}{\cong} H_3(N_{p,q};\La) \cong H_3(\wt{N}_{p,q};\Z) \cong H_3(S^3\times\R ;\Z)\cong \Z.\qedhere
\]
\end{proof}

\begin{prop}\label{prop:TFAE}
	The following are equivalent:
	\begin{enumerate}[(i)]
		\item \label{lensit:1} $N_{p,q}$ and $N_{p,q'}$ are $\CP$-stably diffeomorphic;
		\item \label{lensit:2} $c_*[N_{p,q}]=c_*[N_{p,q'}]\in H_4(\Z/p\times \Z)/\pm \Aut(\Z/p\times \Z)$;
		\item \label{lensit:3} $c_*[L_{p,q}]=c_*[L_{p,q'}]\in H_3(\Z/p)/\pm \Aut(\Z/p)$;
		\item \label{lensit:4} $q\equiv\pm r^2 q'\mod p$ for some $r\in \Z$;
		\item \label{lensit:5} $L_{p,q}$ and $L_{p,q'}$ are homotopy equivalent;
        \item \label{lensit:6} The $\Q/\Z$-valued linking forms of $L_{p,q}$ and $L_{p,q'}$ are isometric.
	\end{enumerate}
\end{prop}
\begin{proof}
	Items \eqref{lensit:1} and \eqref{lensit:2} are equivalent by \cref{theorem:Kreck}.	
By the K\"{u}nneth theorem, \[H_4(\Z/p \times \Z) \cong H_3(\Z/p) \otimes_{\Z} H_1(\Z) \cong H_3(\Z/p) \cong \Z/p.\]  The image of $c_*[L_{p,q}] \in H_3(\Z/p)$ under this identification is precisely $c_*[N_{p,q}]$, since $B\Z = S^1$.
Thus \eqref{lensit:3} and \eqref{lensit:2} are equivalent by the construction of $N_{p,q}$.
	
 Let $m$ be such that $mq=1\mod p$. Then there is a degree $m$ map $g \colon L_{p,q}\to L_{p,1}$ that induces an isomorphism on fundamental groups, given by $[z_1,z_2]\mapsto [z_1,z_2^m]$. Take $c_*[L_{p,1}]$ as the generator of $H_3(\Z/p)$. Then $c_*\circ g_*[L_{p,q}]=m c_*[L_{p,1}]$ corresponds to the element $m =q^{-1} \in \Z/p \cong H_3(\Z/p)$.
 An element $r \in (\Z/p)^\times \cong \Aut(\Z/p)$ acts on $H_1(\Z/p)\cong \Z/p$ and hence also on $H^2(\Z/p)\cong \Ext^1_\Z(H_1(\Z/p),\Z) \cong \Z/p$ by taking the product with $r$. Since $H^*(\Z/p)\cong \Z[x]/(px)$, where the generator $x$ lies in degree two~\cite[Example~3.41]{hatcher}, $r$ acts on $H^4(\Z/p)\cong \Ext^1_\Z(H_3(\Z/p),\Z)$ via multiplication by $r^2$. Therefore, the action of $r$ on $H_3(\Z/p)\cong \Z/p$ is also multiplication by~$r^2$. Hence \eqref{lensit:3} and \eqref{lensit:4} are equivalent.

Items \eqref{lensit:4} and \eqref{lensit:5} are equivalent by \cite[Theorem 10]{Whitehead-lens-spaces}.
Seifert ~\cite{Seifert} computed that the linking form $H_1(L_{p,q}) \times H_1(L_{p,q}) \to \Q/\Z$ is isometric to
\[\begin{array}{rcl} \Z/p \times \Z/p &\to & \Q/\Z  \\  (x,y) & \mapsto & -qxy/p. \end{array}\]
For $q$ and $q'$, the associated forms are isometric up to a sign if and only if \eqref{lensit:4} is satisfied.  Thus \eqref{lensit:4} and \eqref{lensit:6} are equivalent.
\end{proof}

\subsection{The \texorpdfstring{$k$}{k}-invariant is required in general} \label{sec:k-required}

In this section we give examples where the $\CP$-stable classification is determined by the 2-type, but in contrast to the case that $\pi$ is the fundamental group of some aspherical 4-manifold, here the stable  isomorphism class of the second homotopy group is not sufficient to determine the classification.
As with the examples in the previous section, these examples can be compared with  \cite[Conjecture A]{teichnerthesis}, although we remark that this conjecture was only made for finite groups.

Let $X$ be a closed, oriented, aspherical $3$-manifold with fundamental group $G$.
Let $p \geq 2$, let $1 \leq q < p$, and let $X_{p,q}:=L_{p,q}\# X$ and $M_{p,q}:=X_{p,q}\times S^1$. Then $M_{p,q}$ is a closed $4$-manifold with fundamental group $\pi:=(\Z/p*G)\times \Z$.

\begin{lemma}\label{lemma:H1-pi-zero}
The group $\pi$ is infinite and has $H^1(\pi;\La)=0$. In particular, the pair $(\pi_2,k)$ suffices for the $\CP$-stable classification over $\pi$ of oriented manifolds.
\end{lemma}

\begin{proof}
Certainly $\pi = (\Z/p*G)\times \Z$ is infinite. Since $\pi_1(X)=G$ is nontrivial, $\pi_1(X_{p,q}) = \Z/p *G$ is infinite, and hence $\widetilde{X_{p,q}}$ is a noncompact 3-manifold, so $H_3(\widetilde{X_{p,q}};\Z)=0$.
Since $\pi_1(M_{p,q})=\pi$, we see that
\[H^1(\pi;\La)=H^1(M_{p,q};\La)\cong H_3(M_{p,q};\La)\cong H_3(\widetilde{X_{p,q}}\times \R;\Z)\cong H_3(\widetilde{X_{p,q}};\Z)=0.\]
The second sentence of the lemma then follows from \cref{thm:k}~\eqref{thm:k:item2}.
\end{proof}

\begin{lemma}\label{lemma:iso-pi-2}
The second homotopy group $\pi_2(X_{p,q})$ is independent of $q$.
\end{lemma}

\begin{proof}
We show the following statement, from which the lemma follows by taking $Y_1 = L_{p,q}$ and $Y_2=X$: let $Y_1,Y_2$ be closed, connected, oriented 3-manifolds with $\pi_2(Y_1)=\pi_2(Y_2)=0$.  Suppose that $G_1:=\pi_1(Y_1)$ is finite and $G_2:=\pi_1(Y_2)$ is infinite.  Then $\pi_2(Y_1\# Y_2)$ depends only on $G_1$ and $G_2$.

To investigate $\pi_2(Y_1 \# Y_2) = H_2(Y_1 \# Y_2;\Z[G_1*G_2])$, we start by computing $H_2(\cl(Y_i \sm D^3);\Z[G_1 * G_2])$.
For the rest of the proof we write $$R:= \Z[G_1 * G_2].$$
Consider the Mayer-Vietoris sequence for the decomposition $Y_i = \cl(Y_i \sm D^3) \cup_{S^2} D^3$:
\begin{align*}
0 &\ra H_3(Y_i;R)  \\ \xrightarrow{\partial}  H_2(S^2;R) \ra H_2(\cl(Y_i \sm D^3);R) \oplus H_2(D^3;R) &\ra H_2(Y_i;R).
\end{align*}
We have $H_2(Y_i;R) \cong R \otimes_{\Z[G_i]} \pi_2(Y_i) =0$, $H_2(S^2 ;R) \cong R \cong R \otimes_{\Z[G_i]} \Z[G_i]$ and
\[H_3(Y_i;R) = \begin{cases} R \otimes_{\Z[G_i]} \Z & |G_i| < \infty \\ 0 & |G_i| = \infty. \end{cases}\]
In the case that $G_i$ is finite, the boundary map is given by
\[\Id \otimes N \colon  R \otimes_{\Z[G_i]} \Z \ra R \otimes_{\Z[G_i]} \Z[G_i],\]
where $N \colon 1 \mapsto  \sum_{g \in G_i} g$ sends the generator of $\Z$ to the norm element of $\Z[G_i]$.
Thus
\[H_2(\cl(Y_i\sm D^3);R) = \begin{cases} \coker(\Id \otimes N) & |G_i| < \infty \\ R & |G_i| = \infty. \end{cases}\]

Now, as in the hypothesis of the statement we are proving, take $G_1$ to be finite and $G_2$ infinite. Then the Mayer-Vietoris sequence for $Y_1 \# Y_2$,
\begin{align*}
0 &\ra H_3(Y_1 \# Y_2;R)  \\ \to  H_2(S^2;R) \ra H_2(\cl(Y_1 \sm D^3);R) \oplus H_2(\cl(Y_2 \sm D^3);R) &\ra H_2(Y_1 \# Y_2;R) \to 0
\end{align*}
becomes
\begin{align*}
0  \ra  R \xrightarrow{\,j\,} \coker(\Id \otimes N) \oplus R \ra H_2(Y_1 \# Y_2;R) \ra 0.
\end{align*}
From the first Mayer-Vietoris sequence above, in the case that $G_i$ is infinite, we see that the map $R\xrightarrow{j} \coker(\Id \otimes N) \oplus R \xrightarrow{\pr_2} R$ is an isomorphism, where $\pr_2$ is the projection to the $R$ summand.  It follows that
\[\pi_2(Y_1 \# Y_2) \cong H_2(Y_1\# Y_2;R) \cong \coker\big(\Id \otimes N \colon R \otimes_{\Z[G_i]} \Z \ra R \otimes_{\Z[G_i]} \Z[G_i]\big).\]  This $R$-module depends only on the groups $G_1$ and $G_2$, as desired.
\end{proof}

\begin{prop}
For $\pi = (\Z/p*G)\times \Z$ as above, the stable isomorphism class of the pair $(\pi_2,k)$ detects the $\CP$-stable diffeomorphism type of closed, orientable 4-manifolds with fundamental group isomorphic to $\pi$, but the stable isomorphism class of the second homotopy group does not.  That is, there exists a pair of closed $4$-manifolds $M$ and, $M'$ with fundamental group $\pi$ such that $\pi_2(M)$ and $\pi_2(M')$ are stably isomorphic but $M$ and $M'$ are not $\CP$-stably diffeomorphic.
\end{prop}

\begin{proof}
Let $t\colon \Z/p*G\to \Z/p$ be the projection. Then under the induced map $t_*\colon \Omega_3(\Z/p*G)\to \Omega_3(\Z/p)$ the manifold $X_{p,q}$ becomes bordant to $L_{p,q}$, because $\pi_1(X)=G$ maps trivially to $\Z/p$ under $t$.
Cross this bordism with $S^1$ to see that $M_{p,q} = X_{p,q} \times S^1$ is bordant over $\Z/p \times \Z$ to $L_{p,q} \times S^1$.

If the manifolds $M_{p,q}$ and $M_{p,q'}$ are $\CP$-stably diffeomorphic, they are bordant over $\pi$, for some choice of identification of the fundamental groups with $\pi$, because both  $M_{p,q}$ and $M_{p,q'}$ have signature zero.  Therefore the two 4-manifolds are bordant over $\Z/p \times \Z$,.  Combine this with the previous paragraph to see that $L_{p,q} \times S^1$ is bordant to $L_{p,q'} \times S^1$ over $\Z/p \times \Z$.  As we have seen in \cref{section:post-2-type-does-not-suffice-lens-spaces}, this implies that $q=\pm r^2 q'\mod p$ for some $r\in\bbN$ by \cref{prop:TFAE}.  But there are choices of $p, q$ and $q'$ such that this does not hold, so there are pairs of manifolds in the family $\{M_{p,q}\}$ with the same $p$ that are not $\CP$-stably diffeomorphic to one another.

On the other hand, we have $\pi_2(M_{p,q})=\pi_2(X_{p,q}\times S^1)=\pi_2(X_{p,q})$.  By \cref{lemma:iso-pi-2}, $\pi_2(X_{p,q})$ is independent of $q$.
Hence $\pi_2(M_{p,q})\cong \pi_2(M_{p,q'})$ as $\La$-modules for any $q,q'$ coprime to~$p$ with $ 1 \leq q,q' < p$.

We remark that by \cref{lemma:H1-pi-zero}, we know that \cref{thm:k} applies, and so the $k$-invariants must differ for $q$ and $q'$ such that $L_{p,q}$ and $L_{p,q'}$ fail to be homotopy equivalent to one another.
\end{proof}

\bibliographystyle{alpha}
\bibliography{classification}
\end{document}